\documentclass[reqno]{amsart}
\usepackage{amssymb}
\usepackage{tikz}
\usepackage{graphicx}
\usepackage{hyperref}

\newtheorem{theorem}{Theorem}[section]

\newtheorem{lemma}[theorem]{Lemma}

\newtheorem{thm}[theorem]{Theorem}
\newtheorem{defi}[theorem]{Definition}

\newtheorem{prop}[theorem]{Proposition}
\theoremstyle{definition}
\newtheorem{remark}[theorem]{Remark}
\newtheorem{example}[theorem]{Example}

\newcommand{\R}{\mathbb{R}}
\newcommand{\Z}{\mathbb{Z}}
\newcommand{\N}{\mathbb{N}}
\newcommand{\Q}{\mathbb{Q}}
\newcommand{\angbra}[1]{\left\langle #1\right\rangle}
\newcommand{\im}{\operatorname{Im}}

\newcommand{\nbar}[1]{\overline{#1}}
\newcommand{\dcup}{\displaystyle\bigcup}
\newcommand{\dcap}{\displaystyle\bigcap}
\newcommand{\calL}{\mathcal{L}}
\newcommand{\ph}{\varphi}
\newcommand{\eps}{\varepsilon}
\DeclareMathOperator{\id}{id}
\newcommand{\pp}{\mathbb{P}}
\newcommand{\G}{\mathbb{G}}

\newcommand{\Hom}{\operatorname{Hom}}

\newcommand{\calX}{\mathcal{X}}
\newcommand{\calU}{\mathcal{U}}

\newcommand{\into}{\hookrightarrow}
\newcommand{\spec}{\operatorname{Spec}}
\newcommand{\Spec}{\operatorname{Spec}}
\newcommand{\bbT}{\mathbb{T}}

\newcommand{\codim}{\operatorname{codim}}
\newcommand{\vspan}{\mathrm{span}}

\newcommand{\conv}{\operatorname{Conv}}
\newcommand{\rec}{\operatorname{rec}}
\newcommand{\T}{\mathbb{T}}
\newcommand{\calY}{\mathcal{Y}}

\newcommand{\Face}{\operatorname{Face}}
\newcommand{\RelInt}{\operatorname{RelInt}}
\newcommand{\relint}{\operatorname{RelInt}}
\newcommand{\vgamma}{\nbar{\gamma}}
\newcommand{\posloc}{\mathcal{P}}
\newcommand{\ODAG}{\operatorname{ODAG}}
\newcommand{\fncolon}{\colon}
\newcommand{\pairing}{\angbra{\,,}}
\newcommand{\sdrop}{\!\smallsetminus\!}

\begin{document}
\title[Normal completions of toric varieties over rank one valuation rings]{Normal completions of toric varieties over rank one valuation rings and completions of $\Gamma$-admissible fans}
\author{Netanel Friedenberg}
\address{Department of Mathematics, Yale University}
\email{netanel.friedenberg@yale.edu}
\begin{abstract}
We show that any normal toric variety over a rank one valuation ring admits an equivariant open embedding in a normal toric variety which is proper over the valuation ring, after a base-change by a finite extension of valuation rings. If the value group $\Gamma$ is discrete or divisible then no base-change is needed. We give explicit examples which show that existing methods do not produce such normal equivariant completions. Our approach is combinatorial and proceeds by showing that $\Gamma$-admissible fans admit $\Gamma$-admissible completions. In order to show this we prove a combinatorial analog of noetherian reduction which we believe will be of independent interest.
\end{abstract}
\maketitle

\section{Introduction}\label{sec:intro}

Let $v\fncolon K\to\R\cup\{\infty\}$ be a nontrivial valuation on a field $K$ with valuation ring $K^\circ$ and value group $\Gamma$. Let $\bbT$ be a split torus over $K^\circ$ with character lattice $M$ and cocharacter lattice $N$. 
In \cite{Soto} Soto proved that every normal $\bbT$-toric variety $\calX$ admits an equivariant completion, i.e., an equivariant open embedding $\calX\into\nbar{\calX}$ into a $\bbT$-toric variety $\nbar{\calX}$ which is proper over $K^\circ$. 
Unlike in the case of varieties over a field, we cannot immediately conclude that we can find such an embedding with $\nbar{\calX}$ normal, because the normalization of a finite type $K^\circ$-scheme need not be of finite type. 
In particular, we show, by giving explicit examples, that there are projective $\bbT$-toric varieties whose normalizations are not of finite type; see Theorem \ref{thm:BadNormalization}.

In this paper we study normal equivariant completions of normal $\T$-toric varieties. 
Our main result shows that, at least after a finite extension of the base, every normal $\bbT$-toric variety admits such a completion.

\begin{thm}\label{thm:NormalTCompletionsExist}
Let $\calX$ be a normal $\bbT$-toric variety. There are a finite separable totally ramified extension $L/K$ of valued fields, a normal $\bbT_{L^\circ}$-toric variety $\nbar{\calX}$ which is proper over $L^\circ$, and a $\T_{L^\circ}$-equivariant open immersion $\calX_{L^\circ}\into\nbar{\calX}$. 
If $\Gamma$ is discrete or divisible then we may take $L=K$.
\end{thm}

We prove this theorem via a combinatorial approach. In \cite{GublerSoto}, Gubler and Soto classified normal $\bbT$-toric varieties in terms of \textit{$\Gamma$-admissible fans} in $N_{\R}\times\R_{\geq0}$: fans $\Sigma$ in $N_{\R}\times\R_{\geq0}$ such that each cone $\sigma\in\Sigma$ can be written as 
$$\sigma=\{(w,t)\in N_{\R}\times\R_{\geq0}\mid \angbra{u_i,w}+\gamma_i t\geq0\text{ for }i=1,\ldots,\ell\}$$
for some $u_1,\ldots,u_\ell\in M$ and $\gamma_1,\ldots,\gamma_\ell\in\Gamma$, where $N_{\R}:=N\otimes_{\Z}\R$. 
Specifically, they showed that if $\Gamma$ is not discrete then isomorphism classes of normal $\bbT$-toric varieties are in correspondence with $\Gamma$-admissible fans $\Sigma$ in $N_{\R}\times\R_{\geq0}$ such that, for every cone $\sigma\in\Sigma$, all of the vertices of the polyhedron $\sigma\cap(N_{\R}\times\{1\})$ are in $N_{\Gamma}\times\{1\}$. In the case where $\Gamma$ is discrete, it was already shown in \cite[IV, section 3]{ToroidalEmbeddings} that isomorphism types of normal $\bbT$-toric varieties are in bijection with $\Gamma$-admissible fans in $N_{\R}\times\R_{\geq0}$. 
A normal $\T$-toric variety $\calX$ is proper over $K^\circ$ if and only if the corresponding fan $\Sigma$ is complete, i.e., $|\Sigma|=N_{\R}\times\R_{\geq0}$ \cite[Proposition 11.8]{GublerGuideTrop}. Thus much of the content of Theorem \ref{thm:NormalTCompletionsExist} boils down to the following theorem. 

\begin{thm}\label{thm:AdmissibleCompletionsExist}
Any $\Gamma$-admisssible fan in $N_{\R}\times\R_{\geq0}$ has a $\Gamma$-admissible completion.
\end{thm}
\noindent That is, for any $\Gamma$-admissible fan $\Sigma$ in $N_{\R}\times\R_{\geq0}$ there is a complete $\Gamma$-admissible fan $\nbar{\Sigma}$ containing $\Sigma$ as a subfan.

The main tool we use to prove Theorem \ref{thm:AdmissibleCompletionsExist} is the following combinatorial analogue of noetherian reduction, which we believe will be of independent interest and of use in proving other results about $\Gamma$-admissible fans.

\begin{thm}[Combinatorial noetherian reduction]\label{thm:CombinatorialNoetherianReduction}
Let $\Sigma$ be a $\Gamma$-admissible fan in $N_{\R}\times\R_{\geq0}$. For some positive integer $k$ there are a rational fan $\widetilde{\Sigma}$ in $N_{\R}\times\R^k$ and a linear map $\iota\fncolon\R_{\geq0}\to\R^k$ which sends $1$ to a nonzero point of $\Gamma^k$ such that $\Sigma$ is the pullback $(\id_{N_{\R}}\times\iota)^*(\widetilde{\Sigma})$ of $\widetilde{\Sigma}$ along $\id_{N_{\R}}\times\iota$. 

If $\Gamma$ is finitely generated as an abelian group, or, more generally, if the divisible hull $\Q\Gamma$ of $\Gamma$ is finite-dimensional as a $\Q$-vector space, then $k$ and $\iota$ can be chosen independently of $\Sigma$.
\end{thm}

\noindent See Section \ref{sec:preliminaries} for the definition of the pullback of a fan along an injective linear map.

Theorem \ref{thm:CombinatorialNoetherianReduction} allows us to deduce Theorem \ref{thm:AdmissibleCompletionsExist} from the existence of rational completions of rational fans. This fact was first recorded in the literature by Oda. 
On page 18 of \cite{Oda}, Oda pointed out that Sumihiro's equivariant compactification theorem of \cite{Sumihiro}, together with the classification of normal toric varieties over a field, implies that every rational fan admits a rational completion, and mentioned that at that time no combinatorial construction of such completions was known. 
A sketch of a combinatorial proof was first laid out by Ewald in \cite[III, Theorem 2.8]{Ewald} and then filled in by Ewald and Ishida in \cite{EwaldIshida}. 
More recently, in \cite{Rohrer}, Rohrer presented a simplified and improved combinatorial proof based on the ideas of Ewald and Ishida.

Our proof of Theorem \ref{thm:AdmissibleCompletionsExist} uses only Oda's result that rational completions of rational fans exist rather than using ideas from the later combinatorial proofs thereof. 
In fact, it seems very unlikely that the ideas of these combinatorial proofs could be used to prove Theorem \ref{thm:AdmissibleCompletionsExist}. 
One indication of this is that the aforementioned proofs also show that simplicial rational fans admit simplicial rational completions, whereas for $\Gamma$-admissible fans we have the following result.

\begin{thm}\label{thm:FanWOASimplicialCompletion}
If $\Gamma$ has two $\Q$-linearly independent elements then there are simplicial $\Gamma$-admissible fans which do not have a simplicial $\Gamma$-admissible completion.
\end{thm}

We prove Theorem \ref{thm:FanWOASimplicialCompletion} by giving an explicit example of such a fan; see Example \ref{example:FanWOASimplicialCompletion}.

We also show that in Theorem \ref{thm:AdmissibleCompletionsExist} $\R$ cannot be replaced by an arbitrary ordered field $R$. In particular, in Theorem \ref{thm:HigherRankBadFan} we give an explicit example to show that for some ordered fields $R$ there are $R$-admissible fans in $N_{R}\times R_{\geq0}$ which do not have any $R$-admissible completion. In contrast, for any ordered field $R$ every rational fan in $N_{R}$ admits a rational completion; see Proposition \ref{prop:RationalFansInHigherRank}. While our proof of Proposition \ref{prop:RationalFansInHigherRank} uses basic model theory, our proof of Theorem \ref{thm:HigherRankBadFan} is completely elementary.

The structure of the remainder of the paper is as follows. We end the introduction with Example \ref{example:FanWOASimplicialCompletion} in which we prove Theorem \ref{thm:FanWOASimplicialCompletion} and illustrate the method used to prove Theorem \ref{thm:AdmissibleCompletionsExist}. In Section \ref{sec:preliminaries} we set the notations and definitions we will use from polyhedral geometry and prove some elementary and technical results. We then use these in Section \ref{sec:ExistenceOfAdmissibleCompletions} to prove Theorems \ref{thm:AdmissibleCompletionsExist} and \ref{thm:CombinatorialNoetherianReduction}. In Section \ref{sec:CompletionsOverOrderedFields} we consider completions of fans in $N_{R}$ and $N_{R}\times R_{\geq0}$ where $R$ is an arbitrary ordered field and prove the results mentioned in the previous paragraph. Finally, in Section \ref{sec:CompletionsOfTToricVars} we consider $\T$-equivariant completions of $\T$-toric varieties. In particular, we prove Theorem \ref{thm:NormalTCompletionsExist}, give examples of projective $\T$-toric varieties whose normalizations are not of finite type, and show that the $\T$-toric variety corresponding to the fan of Example \ref{example:FanWOASimplicialCompletion} is a semistable $\T$-toric variety which has no semistable $\T$-equivariant completion.

\begin{example}\label{example:FanWOASimplicialCompletion}
Let $\Gamma$ be any additive subgroup of $\R$ which contains two $\Q$-linearly independent elements. So we can find $\alpha,\beta\in\Gamma$ which are $\Q$-linearly independent and both positive. Furthermore, by replacing $\alpha$ and $\beta$ with positive multiples thereof, we can even find such $\alpha$ and $\beta$ with $0<\beta<\alpha<2\beta$. Fix $N=\Z^2$ and consider the $\Gamma$-admissible fan $\Sigma$ in $N_{\R}\times\R_{\geq0}=\R^2\times\R_{\geq0}$ whose maximal cones are
\begin{align*}
\sigma_1:=\{(x,y,t)\in\R^2\times\R_{\geq0}&\mid-3\alpha t\leq x\leq 0, y=0\},\\
\sigma_2:=\{(x,y,t)\in\R^2\times\R_{\geq0}&\mid-3\beta t\leq y\leq 0, x=0\},\\
\sigma_3:=\{(x,y,t)\in\R^2\times\R_{\geq0}&\mid y=2x-3\beta t,0\leq x\leq(\alpha+2\beta)t\},\\
\intertext{and}
\sigma_4:=\{(x,y,t)\in\R^2\times\R_{\geq0}&\mid x=2y-3\alpha t, 0\leq y\leq(2\alpha+\beta)t\}.
\end{align*}
Also, let $\Pi$ be the polyhedral complex, illustrated in Figure \ref{figure:EmptyDart}, consisting of polyhedra $P$ in $\R^2$ such that $P\times\{1\}=\sigma\cap(\R^2\times\{1\})$ for some $\sigma\in\Sigma$.
\begin{figure}[h]
\centering
\begin{tikzpicture}[scale=.25]
\draw (0,0) -- (0,-6) -- (4+pi,2+2*pi) -- (-3*pi,0) --cycle;
\node[below left] at (0,0) {$(0,0)$};
\node[above left] at (-3*pi,0) {$(-3\alpha,0)$};
\node[below right] at (0,-6) {$(0,-3\beta)$};
\node[above] at (4+pi,2+2*pi) {$(\alpha+2\beta,2\alpha+\beta)$};
\end{tikzpicture}
\caption{The polyhedral complex $\Pi$.}
\label{figure:EmptyDart}
\end{figure}
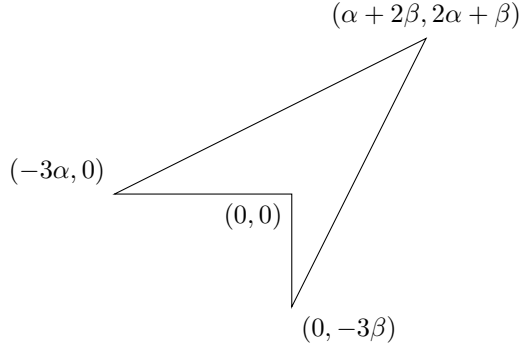

We claim that $\Sigma$ cannot have any simplicial $\Gamma$-admissible completion. To see this, note that such a completion $\nbar{\Sigma}$ would induce a triangulation $\Pi'$ of the dart-shaped non-convex polygon bounded by $|\Pi|$ such that every edge in $\Pi'$ has rational slope. 
Suppose the edges of a triangle $Q\subset\R^2$ all have rational slope, and let $v_0,v_1,$ and $v_2$ be the vertices of $Q$. Let $\ell>0$ be such that $\frac{1}{\ell}(v_1-v_0)$ is a primitive lattice vector. Then $v_0,v_1\in v_0+\ell\Q^2$, and if $v_0,v_1\in v+c\Q^2$ with $v\in\R^2$ and $c\in\R$ then $v+c\Q^2=v_0+\ell\Q^2$. We also have $v_2\in v_0+\ell\Q^2$ because, after shifting and dilating so that $v_0=0$ and $\ell=1$, the coordinates of $v_2$ are given by solving a linear algebra problem over $\Q$. 
Because $\Pi'$ is connected in codimension 1 and $\conv((0,0),(-3\alpha,0))$ is an edge of $\Pi'$, we find that every vertex of $\Pi'$ is in $\alpha\Q^2$. This is impossible since $(0,-3\beta)$ is in $\Pi'$ and $\alpha$ and $\beta$ are $\Q$-linearly independent.

We note that the central idea of the proof above is the same as the idea behind an example of Temkin referred to in \cite[Remark 3.4.2]{TemkinStableModOfRelativeCurves} and \cite[Remark 1.1.1(ii)]{TemkinAlteredUniformization}.

By following our method of proving Theorem \ref{thm:AdmissibleCompletionsExist} we can find some $\Gamma$-admissible completion of $\Sigma$. 
Before applying the aforementioned method, we note that it is not hard to find an extension $\Sigma^*$ of $\Sigma$ such that the induced extension $\Pi^*$ of $\Pi$ has as its support the unbounded region determined by $|\Pi|$. One just needs to choose appropriate rational slopes for rays starting at the vertices of $\Pi$ so that the rays divide this region into polyhedra, let $\Pi^*$ be the polyhedral complex generated by these polyhedra, and let $\Sigma^*$ be the fan over $\Pi^*$. An example of such a $\Pi^*$ is the polyhedral complex generated by the polyhedra labeled $\rho_1,\ldots,\rho_4$ in Figure \ref{figure:CompletedDart}. Thus, to find a completion of $\Sigma$ we only need to find an extension of $\Sigma$ whose induced extension of $\Pi$ has as its support the bounded region determined by $|\Pi|$.

Let $\iota\fncolon\R_{\geq0}\to\R^2$ be the map sending $t\mapsto (t\alpha,t\beta)$. Consider the rational cones 
\begin{align*}
\widetilde{\sigma}_1:=\{(x,y,a,b)\in\R^4&\mid-3a\leq x\leq 0, y=0, \textstyle\frac{1}{2}a\leq b\leq 2a\},\\
\widetilde{\sigma}_2:=\{(x,y,a,b)\in\R^4&\mid-3b\leq y\leq 0, x=0, \textstyle\frac{1}{2}a\leq b\leq 2a\},\\
\widetilde{\sigma}_3:=\{(x,y,a,b)\in\R^4&\mid y=2x-3b,0\leq x\leq a+2b, \textstyle\frac{1}{2}a\leq b\leq 2a\},\\
\intertext{and}
\widetilde{\sigma}_4:=\{(x,y,a,b)\in\R^4&\mid x=2y-3a, 0\leq y\leq 2a+b, \textstyle\frac{1}{2}a\leq b\leq 2a\}
\end{align*}
and note that $(\id_{\R^2}\times\iota)^{-1}(\widetilde{\sigma}_i)=\sigma_i$ for $i=1,2,3,4$. 
By computing generators for $\widetilde{\sigma}_1,\ldots,\widetilde{\sigma}_4$, one can quickly check that $\widetilde{\sigma}_i\cap\widetilde{\sigma}_j=\widetilde{\sigma}_i\cap\vspan(\widetilde{\sigma}_j)$ for all $i$ and $j$, which shows that $\widetilde{\sigma}_1,\ldots,\widetilde{\sigma}_4$ are the maximal cones of a fan $\widetilde{\Sigma}$. 
The defining inequalities of the $\widetilde{\sigma}_i$s show that each $\widetilde{\sigma}_i$ is contained in the half-space $\{(x,y,a,b)\in\R^4\mid b\geq0\}$, and no $\widetilde{\sigma}_i$ is contained in the hyperplane $\{(x,y,a,b)\in\R^4\mid b=0\}$, as can be quickly verified because $(\id_{\R^2}\times\iota)^{-1}(\widetilde{\sigma}_i)=\sigma_i$. Hence the geometry of $\widetilde{\Sigma}$ is completely described by the restriction $\widetilde{\Pi}$ of $\widetilde{\Sigma}$ to the affine hyperplane $H=\{(x,y,a,1)\in\R^4\}$. 
The polyhedral complex $\widetilde{\Pi}$, as seen from two different perspectives, is illustrated in Figure \ref{figure:Picture3DDart}. 

\begin{figure}[h]
\centering
\includegraphics{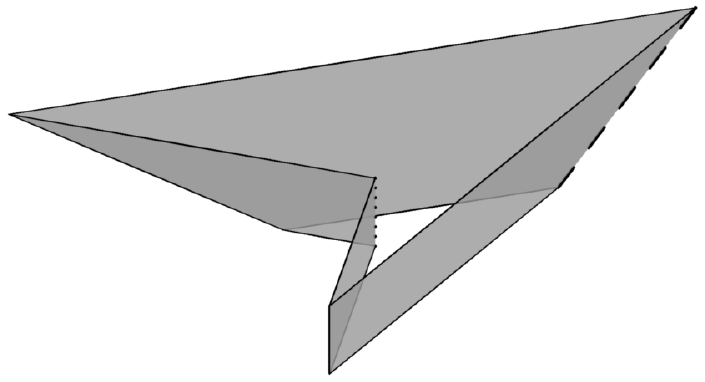}
\includegraphics{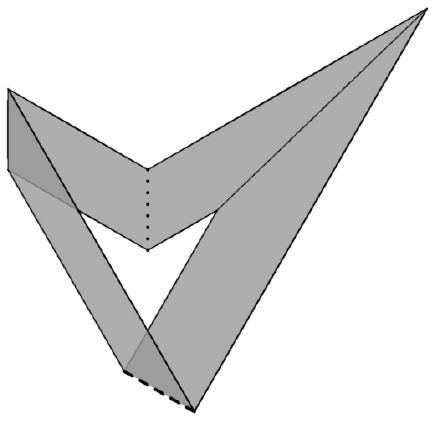}
\caption{The polyhedral complex $\widetilde{\Pi}$ from two different perspectives.}
\label{figure:Picture3DDart}
\end{figure}

We see that $\widetilde{\Pi}$ forms part of the boundary of a nonconvex polytope $\Xi$ whose two other faces are darts. 
We first describe the partial completion $\widetilde{\Sigma}'$ of $\widetilde{\Sigma}$ by describing its restriction $\widetilde{\Pi}'$ to $H$, which is a subdivision of $\Xi$. 
Note that the edges $\widetilde{\sigma}_1\cap\widetilde{\sigma}_2\cap H$ and $\widetilde{\sigma}_3\cap\widetilde{\sigma}_4\cap H$, which are dotted and dashed, respectively, in Figure \ref{figure:Picture3DDart}, are skew. Hence we cannot use a plane containing both of the aforementioned edges to split $\Xi$ into two convex polytopes, because there is no such plane. Instead, we subdivide $\Xi$ as follows. 
First, divide each of the dart-shaped faces of $\Xi$ into two triangles and let $\widetilde{\Pi}''$ be the polyhedral complex generated by the aforementioned triangles and $\widetilde{\Pi}$. 
Then let $\widetilde{\Pi}'$ be the join of $\widetilde{\Pi}''$ with the point $(1,1,1,1)$, which is a polyhedral complex because $\Xi$ is star-shaped around $(1,1,1,1)$. 
Let $\widetilde{\Sigma}'$ be the fan over $\widetilde{\Pi}'$. 

Let $\widetilde{\tau}_1,\ldots,\widetilde{\tau}_8$ be the maximal cones of $\widetilde{\Sigma}'$ labeled so that $\widetilde{\sigma}_i$ is a face of $\widetilde{\tau_i}$ for $i=1,\ldots,4$ and $\widetilde{\tau}_5,\widetilde{\tau}_6,\widetilde{\tau}_7,$ and $\widetilde{\tau}_8$ are simplicial. Let $\tau_i:=(\id_{\R^2}\times\iota)^{-1}(\widetilde{\tau}_i)$ for $i=1,\ldots, 8$. Two of $\tau_5,\tau_6,\tau_7$, and $\tau_8$ are the zero cone $\{(0,0,0)\}$; by relabeling we may take these to be $\tau_7$ and $\tau_8$. The cones $\tau_1,\ldots,\tau_6$ are the maximal cones of a $\Gamma$-admissible fan $\Sigma'$ in $\R^2\times\R_{\geq0}$ which is an extension of $\Sigma$. The desired completion of $\Sigma$ is $\nbar{\Sigma}:=\Sigma^*\cup\Sigma'$. The induced completion of the polyhedral complex $\Pi$ is illustrated in Figure \ref{figure:CompletedDart}, with each maximal face labeled by the corresponding maximal cone of $\nbar{\Sigma}$.
\begin{figure}[h]
\centering
\begin{tikzpicture}[scale=.7]
\draw (0,0) -- (0,-6) -- (4+pi,2+2*pi) -- (-3*pi,0) --cycle;
\fill[gray!10!white] (-3*pi-1,-7) rectangle (4+pi+1,2+2*pi+1);
\draw (-3*pi,0)--(0,0)--(4-pi,4-pi)--(16-7*pi,4-pi)--cycle;
\node at (8-3.5*pi,2-pi/2) {$\tau_1$};
\draw (0,0)--(0,-6)--(4-pi,10-4*pi)--(4-pi,4-pi)--cycle;
\node at (2-pi/2,5-2*pi) {$\tau_2$};
\draw (0,-6)--(4+pi,2+2*pi)--(3*pi-4,4*pi-6)--(4-pi,10-4*pi)--cycle;
\node at (1+3*pi/4,pi/2) {$\tau_3$};
\draw (4+pi,2+2*pi)--(-3*pi,0)--(16-7*pi,4-pi)--(3*pi-4,4*pi-6)--cycle;
\node at (4-1.5*pi,5*pi/4) {$\tau_4$};
\draw (4-pi,4-pi)--(3*pi-4,4*pi-6);
\node at (4-5*pi/4,.5+.5*pi) {$\tau_5$};
\node at (1+pi/4,2-pi/4) {$\tau_6$};
\draw(4+pi,2+2*pi)--(4+pi+1,2+2*pi+1);
\draw(0,-6)--(-.25,-7);
\draw(-3*pi,0)--(-3*pi-1,-.25);
\draw(0,0)--(-7,-7);
\node at (-9*pi/5-9/5,-14.25/5) {$\rho_1$};
\node at (-7.25/4,-20/4) {$\rho_2$};
\node at (13.75/5+3*pi/5,-15/5+4*pi/5) {$\rho_3$};
\node at (7/5-7*pi/5,7.75/5+6*pi/5) {$\rho_4$};
\node[below] at (-1.25,-2.5) {$(0,0)$};
\draw[->, gray] (-1.25,-2.5) -- (-.1,-.2);
\node[left] at (-2.5,-1.25) {$(2\beta-\alpha,2\beta-\alpha)$};
\draw[->, gray] (-2.5,-1.25) -- (4-pi-.2,4-pi-.1);
\node[right] at (0,-6) {$(0,-3\beta)$};
\node[below right] at (4-pi+1,10-4*pi-1) {$(2\beta-\alpha,5\beta-4\alpha)$};
\draw[->, gray] (4-pi+1,10-4*pi-1) -- (4-pi+.1,10-4*pi-.1);
\node[below] at (3*pi-4+1-.1,4*pi-6-2-2) {$(\alpha+2\beta,2\alpha+\beta)$};
\draw[->, gray] (3*pi-4+1-.1,4*pi-6-2-2) -- (4+pi,2+2*pi-.2);
\node[left] at (3*pi-4-2,4*pi-6+1) {$(3\alpha-2\beta,4\alpha-3\beta)$};
\draw[->, gray] (3*pi-4-2,4*pi-6+1) -- (3*pi-4-.2,4*pi-6+.1);
\node[above] at (16-7*pi,4-pi+2) {$(8\beta-7\alpha,2\beta-\alpha)$};
\draw[->, gray] (16-7*pi,4-pi+2) -- (16-7*pi,4-pi+.2);
\node[above] at (-3*pi,0+5) {$(-3\alpha,0)$};
\draw[->, gray] (-3*pi,0+5) -- (-3*pi,0+.3);
\end{tikzpicture}
\caption{The completion of $\Pi$ induced by $\nbar{\Sigma}$.}
\label{figure:CompletedDart}
\end{figure}
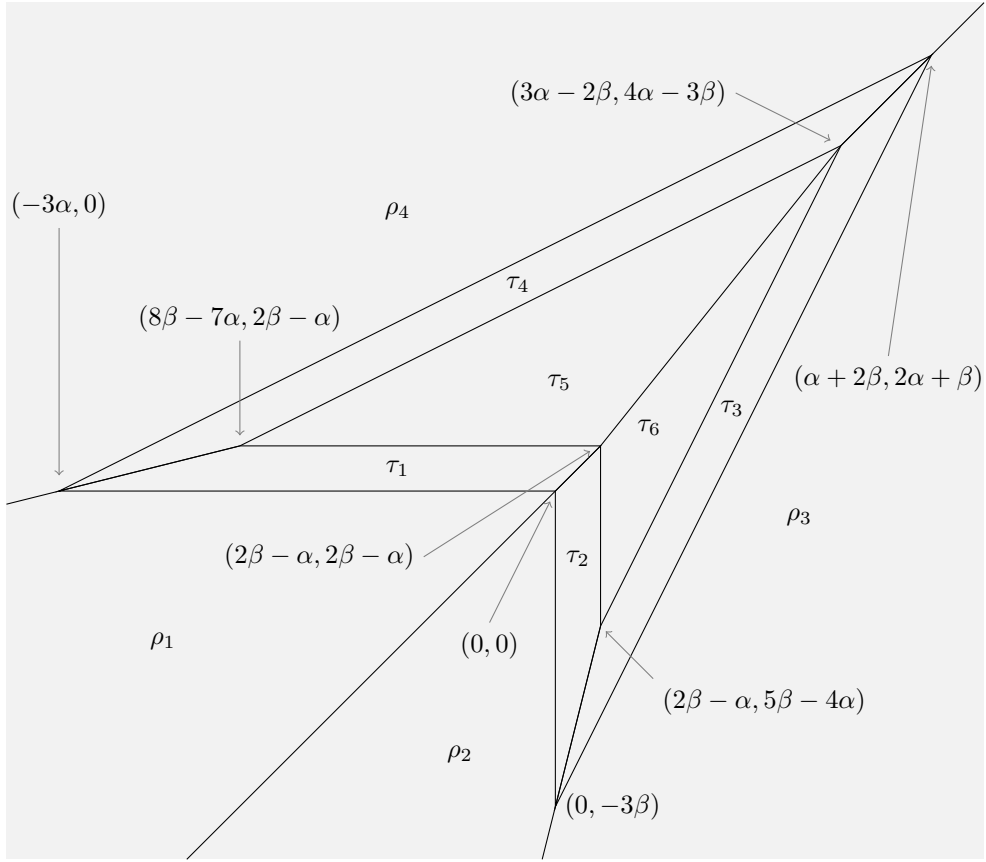
\end{example}

\subsection*{Acknowledgments} 
I am particularly thankful to Sam Payne for his guidance throughout my work on this project. I am also grateful to Daniel Corey, Dhruv Ranganathan, Michael Temkin, and Jeremy Usatine for helpful discussions. I am thankful for the hospitality of the Fields Institute for Research in Mathematical Sciences and to the organizers of the Major Thematic Program on Combinatorial Algebraic Geometry during which the work on this project began. 
I also thank the Mathematical Sciences Research Institute and the organizers of the Birational Geometry and Moduli Spaces program in the Spring 2019 semester during which part of the work on this project was carried out. The work done during that program was supported by NSF Grant DMS-1440140.

%
%

\section{Preliminaries}\label{sec:preliminaries}

In this section we fix notations and establish some basic facts about polyhedra and polyhedral complexes. Instead of considering polyhedra in a vector space over the real numbers $\R$ we will work in the greater generality of polyhedra in a vector space over an arbitrary ordered field $R$. For most of this paper, and in particular for the proofs of Theorems \ref{thm:NormalTCompletionsExist}, \ref{thm:AdmissibleCompletionsExist}, and \ref{thm:CombinatorialNoetherianReduction}, we will only need the cases where $R$ is either $\R$ or the field $\Q$ of rational numbers. 
We will only use the more general setup in Section \ref{sec:CompletionsOverOrderedFields}, which is independent of the rest of the paper.
We note that the usual basic theory of polyhedra works in this full generality with the exception of results about simple or simplicial polyhedra which use the density of $\Q$ in $\R$ to reduce to the case where the polyhedra are rational. 
In particular, all of Chapter 1 and Sections 2.1 through 2.4 of \cite{ZieglerPolytopes} still applies word-for-word when $\R$ is replaced by an arbitrary ordered field $R$.

Let $N$ be a lattice, let $M:=\Hom(N,\Z)$ be its dual lattice, and let $R/S$ be an extension of ordered fields. We let $\pairing$ denote the natural pairing $M\times N\to\Z$, and also its extension $M_{R}\times N_{R}\to R$ where $M_{R}:=M\otimes_{\Z}R$ and $N_{R}:=N\otimes_{\Z}R$. When we consider $N_{S}$ as a subset of $N_{R}$ we will call $N_{S}$ the set of $S$-rational points. A polyhedron in $N_{R}$ is called \emph{$S$-definable} if it can be written as an intersection of finitely many half-spaces of the form $\{w\in N_{R}\mid \angbra{u,w}\geq a\}$ with $u\in M_{S}$ and $a\in S$. A polyhedron which is $\Q$-definable may also be called \emph{rational}. It is clear that given a finite set of inequalities of the form $\angbra{u,w}\geq a$ with $u\in M_S$ and $a\in S$, the polyhedron they define in $N_{S}$ is the same as the set of $S$-rational points of the corresponding $S$-definable polyhedron in $N_{R}$. 
Part (\ref{lemmaPart:PosKLinearSolution}) of the following lemma shows a similar result for the cone generated by finitely many elements of $N_{S}$.

\begin{lemma}\label{lemma:KPointsExist}
Let $R/S$ be an extension of ordered fields.
\begin{enumerate}
\item\label{lemmaPart:KPointPolyhedron} A $S$-definable polyhedron $P$ in $N_{R}$ is nonempty if and only if $P$ contains a $S$-rational point.
\item\label{lemmaPart:PosKLinearSolution} Say $w_0,w_1,\ldots,w_m\in N_{S}$. Then there exist $b_1,\ldots, b_m\in R_{\geq0}$ such that $\sum_{i=1}^m b_iw_i=w_0$ if and only if there exist $b_1',\ldots,b_m'\in S_{\geq0}$ such that $\sum_{i=1}^m b_i'w_i=w_0$. 
\end{enumerate}
\end{lemma}
\begin{proof}
(1) The ``if'' direction is trivial, so suppose that $P$ contains no $S$-rational point. Writing $P$ as $P=\{w\in N_{R}\mid(\forall i=1,\ldots,n)\angbra{u_i,w}\leq a_i\}$ with $u_i\in M_S$ and $a_i\in S$, the generalization of \cite[Proposition 1.7]{ZieglerPolytopes} to polyhedra over $S$ gives us that there are $c_1,\ldots,c_n\in S_{\geq0}$ such that $\sum_{i=1}^nc_iu_i=0$ and $\sum_{i=1}^nc_ia_i<0$. Then interpreting the $c_i$s as being in $R_{\geq0}$, the trivial direction of the same proposition shows that $P$ has no $R$-rational points either.

(2) Consider the polyhedron $P:=\{(b_1,\ldots,b_m)\in R^m\mid b_i\geq 0,\sum_{i=1}^mb_iw_i=w_0\}$ in $R^m$. Note that $P$ is $S$-definable as it is the intersection of the cone $(R_{\geq0})^n$ with the $S$-definable affine subspace $\{(b_1,\ldots,b_m)\in R^m\mid\sum_{i=1}^mb_iw_i=w_0\}$. Applying part (1) to $P$ gives the result.
\end{proof}

The proof of Theorem \ref{thm:CombinatorialNoetherianReduction} will use two main technical lemmas. The first of these lemmas is a slight strengthening of a well-known version of the Farkas Lemma which will be deduced from part (\ref{lemmaPart:PosKLinearSolution}) of Lemma \ref{lemma:KPointsExist}. Before proving this lemma we recall a classical form of the Farkas Lemma.

\begin{prop}\label{prop:UsualFarkas}
Let $u_0,u_1,\ldots,u_m\in M_{R}$ and $a_0,a_1,\ldots,a_m\in R$. Suppose that $P:=\{x\in N_{R}\mid(\forall i=1,\ldots,m)\angbra{u_i,x}\geq a_i\}$ is nonempty. Then $\angbra{u_0,x}\geq a_0$ is valid for all $x\in P$ if and only if there exist $c_1,\ldots,c_m\in R_{\geq0}$ such that $\sum_{i=1}^mc_iu_i=u_0$ and $\sum_{i=1}^mc_ia_i\geq a_0$.
\end{prop}\begin{proof}
This is easily seen to be equivalent to the extension of \cite[Proposition 1.9]{ZieglerPolytopes} to an arbitrary ordered field. Note that in both statements the content is in the ``only if'' direction.
\end{proof}

\begin{prop}\label{prop:FarkasPlus}
Let $S$ be any subfield of $R$. Let $u_0,u_1,\ldots,u_m\in M_{S}$ and $a_0,a_1,\ldots,a_m\in R$. Suppose that $P:=\{x\in N_{R}\mid(\forall i=1,\ldots,m)\angbra{u_i,x}\geq a_i\}$ is nonempty. Then $\angbra{u_0,x}\geq a_0$ is valid for all $x\in P$ if and only if there are $c_1,\ldots,c_m\in S_{\geq0}$ such that $\sum_{i=1}^m c_iu_i=u_0$ and $\sum_{i=1}^m c_ia_i\geq a_0$. 
\end{prop}
\begin{proof}
The ``if'' direction is clear.

For the ``only if'' direction, note that $u_0$ is bounded below on the polyhedron $P$, so there is some (nonempty) face $F$ of $P$ on which $u_0$ is minimized. Let $\delta:=u_0(F)$, so $\delta\geq a_0$ and $\angbra{u_0,x}\geq \delta$ is valid for all $x\in P$. By Proposition \ref{prop:UsualFarkas} there are $c_1',\ldots,c_m'\in R_{\geq0}$ such that $\sum_{i=1}^mc_i'u_i=u_0$ and $\sum_{i=1}^m c_i'a_i\geq\delta$. Moreover, because $\angbra{u_0,F}=\delta$ we have that $\sum_{i=1}^mc_i'a_i=\delta$ and letting $I:=\{i\in\{1,\ldots,m\}\mid\angbra{u_i,F}=a_i\}$ we get $\{i\in\{1,\ldots,m\}\mid c_i'\neq0\}\subset I$. So by part (\ref{lemmaPart:PosKLinearSolution}) of Lemma \ref{lemma:KPointsExist} there are $c_i\in S_{\geq0}$ for $i\in I$ such that $\sum_{i\in I}c_iu_i=u_0$. Furthermore, for any $x\in F$ we have $\delta=\angbra{u_0,x}=\sum_{i\in I}c_i\angbra{u_i,x}=\sum_{i\in I}c_ia_i$. So letting $c_i=0$ for $i\notin I$, we are done.
\end{proof}

Before stating the second technical lemma needed in the proof of Theorem \ref{thm:CombinatorialNoetherianReduction} we review some notation. 
For $u\in M_{R}$ let $u^\vee:=\{w\in N_{R}\mid \angbra{u,w}\geq0\}$ and $u^\perp:=\{w\in N_{R}\mid \angbra{u,w}=0\}$. 
Extending this to sets, for $L\subset M_R$ we have $L^\vee:=\dcap_{u\in L}u^\vee$ and $L^\perp:=\dcap_{u\in L}u^\perp$. These definitions immediately extend to any case where we have a pairing $\pairing$ between two sets taking values in some ordered abelian group. A \emph{rational cone} in $N_{R}$ is a cone which can be written as $L^\vee$ for some finite $L\subset M\subset M_{R}$.

\begin{lemma}\label{lemma:ThinRationalConeAroundPoint}
Let $u\in M$ and say $w\in N_{\R}$ is in $u^\vee\sdrop u^\perp$. Then there is a rational cone $\sigma\subset N_{\R}$ such that $w\in \sigma$, $\sigma\subset u^\vee$, and $\sigma\cap u^\perp=\{0\}$.
\end{lemma}\begin{proof}
Because $w\in u^\vee\sdrop u^\perp$, we must have $u\neq 0$, so we can extend $u$ to a basis $\{u_1=u,u_2,\ldots,u_n\}$ of $M_{\Q}$. 
Let $w'=\frac{1}{\angbra{u,w}}w$, so $w'\in\{z\in N_{\R}\mid \angbra{u,z}=1\}$. By picking a closed rational interval $I_i$ containing $\angbra{u_i,w'}$ for each $i\in\{2,\ldots,n\}$ we get a rational $(n-1)$-box $P:=\{z\in N_{\R}\mid \angbra{u,z}=1,\angbra{u_i,z}\in I_i(\forall i\in\{2,\ldots,n\})\}$ containing $w'$. Because $P$ is a rational polytope in $\{z\in N_{\R}\mid \angbra{u,z}=1\}$, we have that $\sigma:=\{tz\mid t\in\R_{\geq0}, z\in P\}$ is a rational cone contained in $u^\vee$ with $\sigma\cap u^\perp=\{0\}$. Finally, $\sigma$ is a cone and contains $w'$, so it also contains the positive multiple $w$ of $w'$.
\end{proof}

Lemma \ref{lemma:ThinRationalConeAroundPoint} is false if $\R$ is replaced by an arbitrary ordered field.

\begin{example}
Let $R$ be a non-archimedean ordered field, i.e., an ordered field such that there is some $\omega\in R$ with $\omega>n$ for all $n\in\N$. Let $M=\Z^2$ and fix $u=(1,0)\in M$ and $w=(\eps,1)\in N_{R}$ where $\eps=1/\omega$. We have $0<\eps<\frac{1}{n}$ for all positive integers $n$ and so $0<\eps<q$ for all positive $q\in\Q$. In particular, $w\in u^\vee\sdrop u^\perp$. Consider a rational cone $\sigma$ in $N_R$ such that $w\in \sigma\subset u^\vee$. Because $\sigma$ is rational and contained in $u^\vee$, $\{t\in R\mid (t,1)\in \sigma\}$ is either of the form $\{t\in R\mid a\leq t\leq b\}$ for some nonnegative $a,b\in\Q$ or of the form $\{t\in R\mid a\leq t\}$ for some nonnegative $a\in\Q$. In either case, because $(\eps,1)\in \sigma$ and $\eps<q$ for all positive $q\in\Q$ we must have $a=0$. So we have $(0,1)\in \sigma\cap u^\perp$ and hence $\sigma\cap u^\perp\neq\{0\}$.
\end{example}

If $P$ is a polyhedron in $N_R$ and $F$ is a face of $P$ then we will write $F\leq P$. Recall that a \emph{polyhedral complex} $\Pi$ in $N_R$ is a finite collection of polyhedra in $N_R$ such that 
\begin{enumerate}
\item if $P\in\Pi$ and $F\leq P$ then $F\in\Pi$ and
\item if $P,Q\in\Pi$ then $P\cap Q$ is either empty or a face of both $P$ and $Q$.
\end{enumerate}
The polyhedra $P\in \Pi$ are called the faces of $\Pi$. The relation $\leq$ makes $\Pi$ into a partially ordered set. Let $\Pi_{\max}$ denote the set of maximal faces of $\Pi$. From (1) we know that $\Pi$ is determined by $\Pi_{\max}$ via $\Pi=\dcup_{P\in\Pi_{\max}}\Face(P)$, where, for any polyhedron $P$, $\Face(P)$ is the set of faces of $P$. The following lemma classifies which sets of polyhedra can generate a polyhedral complex in this way. Although the statement is well-known to experts, we were not able to easily find a reference to it in the literature, and so include a brief proof.

\begin{lemma}\label{lemma:ComplexFromBigPolyhedra}
Let $\Phi$ be a finite set of polyhedra in $N_{R}$. Suppose that for any $P,Q\in\Phi$, $P\cap Q$ is either empty or a face of both $P$ and $Q$. Then $\dcup_{P\in\Phi}\Face(P)$ is a polyhedral complex.
\end{lemma}\begin{proof}
We prove this lemma in the case when $\Phi$ consists of cones, which is the case in which we will use the result. The proof in the general case is analogous. 
Let $\Sigma=\dcup_{\sigma\in\Phi}\Face(\sigma)$. Obviously every face of any cone in $\Sigma$ is again in $\Sigma$. For $\sigma',\tau'\in\Sigma$ there are $\sigma,\tau\in\Phi$ such that $\sigma'\leq\sigma$ and $\tau'\leq\tau$. In particular there are $u_\sigma,u_\tau\in M_R$ such that $\sigma\subset u_\sigma^\vee$, $\sigma'=\sigma\cap u_\sigma^\perp$, $\tau\subset u_\tau^\vee$, and $\tau'=\tau\cap u_\tau^\perp$. Now 
\begin{align*}
\sigma'\cap\tau'&=(\sigma\cap u_\sigma^\perp)\cap(\tau\cap u_\tau^\perp)\\
&=(\sigma\cap\tau\cap u_\sigma^\perp)\cap(\sigma\cap\tau\cap u_\tau^\perp)
\end{align*}
with $\sigma\cap\tau$ contained in both $u_\sigma^\vee$ and $u_\tau^\vee$, so we get that $\sigma'\cap\tau'$ is a face of $\sigma\cap\tau$ because an intersection of faces of a cone is again a face of the same cone. Because $\sigma\cap\tau$ is a face of both $\sigma$ and $\tau$ we get that $\sigma'\cap\tau'$ is a face of both $\sigma$ and $\tau$. So $\sigma'\cap\tau'$ is a face of $\sigma$ which is contained in the face $\sigma'$ of $\sigma$, so $\sigma'\cap\tau'$ is a face of $\sigma'$. Similarly, we find that $\sigma'\cap\tau'$ is a face of $\tau'$.
\end{proof}

The \emph{support} of a polyhedral complex $\Pi$ is $|\Pi|:=\dcup_{P\in\Pi}P$. If $\Pi$ is a polyhedral complex in $N_{R}$ then $\Pi$ is called \emph{complete} if $|\Pi|=N_{R}$.

Recall that a \emph{generalized fan} in $N_{R}$ is a nonempty polyhedral complex consisting of cones. A cone is called \emph{strongly convex} or \emph{pointed} if it contains no line. A \emph{fan} is a nonempty polyhedral complex consisting of strongly convex cones or, equivalently, a generalized fan containing the cone $\{0\}$. A fan is \emph{rational} if it consists of rational cones.

A fan in $N_R\times R_{\geq0}$ is a fan in $N_{R}\times R$ which is supported in the half-space $N_{R}\times R_{\geq0}$. A fan $\Sigma$ in $N_R\times R_{\geq0}$ is called \emph{complete} if $|\Sigma|=N_R\times R_{\geq0}$. 
Let $\Gamma$ be an additive subgroup of $R$. 
A cone in $N_{R}\times R_{\geq0}$ is called \emph{$\Gamma$-admissible} if it is strongly convex and can be written as $L^\vee$ for a finite set $L\subset M\times\Gamma$ where we use the extended pairing $\pairing\fncolon(M\times\Gamma)\times(N_R\times R_{\geq0})\to R$ given by $\angbra{(u,\gamma),(w,t)}=\angbra{u,w}+\gamma t$. 
If we replace $M\times\Gamma$ with $M_{\Q}\times\Q\Gamma$ we get the same notion of $\Gamma$-admissible cones because we can clear denominators from any inequality of the form $\angbra{\frac{u}{r},w}+t\frac{\gamma}{s}\geq0$ with $u\in M, \gamma\in\Gamma$ and $r,s\in\Z_{>0}$ to get the equivalent inequality $\angbra{su,w}+tr\gamma\geq0$. 
A polyhedron in $N_{R}$ is called \emph{$\Gamma$-rational} if it can be written as an intersection of finitely many half-spaces of the form $\{w\in N_{R}\mid\angbra{u,w}\geq\gamma\}$ for some $u\in M$ and $\gamma\in\Gamma$. A strongly convex cone $\sigma$ in $N_{R}\times R_{\geq0}$ which meets $N_{R}\times\{1\}$ is $\Gamma$-admissible if and only if the polyhedron $\{w\in N_{R}\mid (w,1)\in\sigma\}$ is $\Gamma$-rational. 
If $\sigma$ does not meet $N_R\times\{1\}$ then $\sigma$ is contained in $N_{R}\times\{0\}$ in which case $\sigma$ is $\Gamma$-admissible if and only if $\sigma$ is rational. 
A fan in $N_{R}\times R_{\geq0}$ is called \emph{$\Gamma$-admissible} if it consists of $\Gamma$-admissible cones.
Part (\ref{lemmaPart:SeparationOfCones}) of the following lemma shows that any pair of cones in a $\Gamma$-admissible fan can be separated by a hyperplane given as the vanishing locus of some $u\in M\times\Gamma$.

\begin{lemma}[Separation Lemma]\label{lemma:AdmissibleSeparationLemma}
\ 
\begin{enumerate}
\item\label{lemmaPart:SeparationOfPolyhedra} Let $P$ and $Q$ be $\Gamma$-rational polyhedra in $N_{R}$ such that $P\cap Q$ is either empty or a face of both $P$ and $Q$. Then there are $u\in M$ and $\gamma\in\Gamma$ such that $\angbra{u,w}\geq\gamma$ for $w\in P$, $\angbra{u,w}\leq\gamma$ for $w\in Q$, and $P\cap \{w\in N_{R}\mid\angbra{u,w}=\gamma\}=P\cap Q=Q\cap\{w\in N_{R}\mid \angbra{u,w}=\gamma\}$.
\item\label{lemmaPart:SeparationOfCones} Let $\sigma$ and $\tau$ be $\Gamma$-admissible cones in $N_{R}\times R_{\geq0}$ such that $\sigma\cap\tau$ is a face of both $\sigma$ and $\tau$. Then there is $y\in M\times\Gamma$ such that $\sigma\subset y^\vee$, $\tau\subset (-y)^\vee$, and $\sigma\cap y^\perp=\sigma\cap\tau=\tau\cap y^\perp$.
\end{enumerate}
\end{lemma}\begin{proof}
(\ref{lemmaPart:SeparationOfPolyhedra}) As shown in the proof of \cite[Lemma 7.8]{GublerGuideTrop} there are $u\in M$ and $w_0\in N_{\Q\Gamma}$ such that $P\subset u^\vee+w_0$, $Q\subset(-u)^\vee+w_0$, and $P\cap(u^\perp+w_0)=P\cap Q=Q\cap(u^\perp+w_0)$. Letting $\gamma=\angbra{u,w_0}\in\Q\Gamma$ we have
\begin{align*}
w\in u^\vee+w_0&\text{ if and only if }\angbra{u,w}\geq\gamma,\\
w\in u^\perp+w_0&\text{ if and only if }\angbra{u,w}=\gamma,\\
\intertext{and}
w\in (-u)^\vee+w_0&\text{ if and only if }\angbra{u,w}\leq\gamma.
\end{align*}
The result follows upon clearing denominators to get $\gamma\in\Gamma$.

(\ref{lemmaPart:SeparationOfCones}) We consider three cases as to the positions of the cones. If both $\sigma$ and $\tau$ meet $N_{R}\times\{1\}$ then the claim follows by applying part (1) to the $\Gamma$-rational polyhedra $P$ and $Q$ in $N_R$ such that $P\times\{1\}=\sigma\cap(N_R\times\{1\})$ and $Q\times\{1\}=\tau\cap(N_R\times\{1\})$ and homogenizing. 
If $\sigma$ and $\tau$ are both contained in $N_R\times\{0\}$ then by considering $\sigma$ and $\tau$ as rational cones in $N_R$ we can get such a $y\in M\times\{0\}$ by \cite[Lemma 1.2.13]{CLS}. 
Now suppose that one of $\sigma$ or $\tau$ is contained in $N_R\times\{0\}$ and the other is not. We may assume without loss of generality that $\sigma\subset N_R\times\{0\}$ while $\tau\not\subset N_R\times\{0\}$. Because $\sigma\subset N_R\times\{0\}$ we must have that $\Gamma$ is nontrivial. Let $\sigma_0$ and $\tau_0$ be the intersections of $\sigma$ and $\tau$, respectively, with $N_{R}\times\{0\}$, viewed as rational cones in $N_R$. By \cite[Lemma 1.2.13]{CLS}, there is $u\in M$ with $\sigma_0\subset u^\vee$, $\tau_0\subset(-u)^\vee$, and $\sigma_0\cap u^\perp=\sigma_0\cap\tau_0=\tau_0\cap u^\perp$. Letting $P$ be the $\Gamma$-rational polyhedron in $N_R$ such that $P\times\{1\}=\tau\cap(N_R\times\{1\})$ we have that the recession cone of $P$ is $\rec P=\tau_0$. Since $\tau_0\subset (-u)^\vee$ this gives us that $u$ is bounded above on $P$. The maximum $\gamma$ of $u$ on $P$ is achieved at some vertex $v$ of $P$, and because $v\in N_{\Q\Gamma}$ this gives that $\gamma\in\Q\Gamma$. Let $m$ be a positive integer such that $m\gamma\in\Gamma$ and let $\delta$ be any element of $\Gamma$ which is greater than $m\gamma$. 
Letting $y=(mu,-\delta)$ we have $y^\vee\cap(N_R\times\{0\})=(mu,0)^\vee\cap(N_R\times\{0\})=(u,0)^\vee\cap(N_R\times\{0\})\supset\sigma$ and $\sigma\cap y^\perp=\sigma\cap(N_R\times\{0\})\cap y^\perp=\sigma\cap(u,0)^\perp=\sigma\cap\tau$. 
Because $\angbra{mu,x}<\delta$ for $x\in P$ we have $\tau\subset(-y)^\vee$ and $\tau\cap(-y)^\perp\subset N_R\times\{0\}$ and so also $\tau\cap y^\perp=\tau\cap y^\perp\cap (N_R\times\{0\})=\tau\cap (u,0)^\perp\cap (N_R\times\{0\})=\tau\cap\sigma$.
\end{proof}

If $\Gamma$ is not the trivial group then every face of a $\Gamma$-admissible cone is $\Gamma$-admissible. On the other hand, every $\{0\}$-admissible cone is of the form $\sigma\times\R_{\geq0}$ for some rational cone $\sigma\subset N_R$ and so has the face $\sigma\times\{0\}$ which is not $\{0\}$-admissible. In particular, there are no $\{0\}$-admissible fans.

A \emph{linear map} $\ph\fncolon N_{R}\times R_{\geq0}\to W$ with $W$ an $R$-vector space is a function which is the restriction of a linear map $N_R\times R\to W$ to the half-space $N_{R}\times R_{\geq0}$.

\begin{defi}
Let $W$ be an $R$-vector space and let $\ph\fncolon N_{R}\times R_{\geq0}\to W$ be an injective linear map. For any fan $\Sigma$ in $W$, the \emph{pullback of $\Sigma$ along $\ph$} is the fan
$$\ph^*(\Sigma):=\{\ph^{-1}(\sigma),(\ph|_{N_{R}\times\{0\}})^{-1}(\sigma)\mid\sigma\in\Sigma\}.$$
\end{defi}

To see that $\ph^*(\Sigma)$ is a fan note that, because $\ph$ gives a linear isomorphism $N_{R}\times R_{\geq0}\to\im(\ph)$, it suffices to show that $\Delta:=\{\sigma\cap\im(\ph),\sigma\cap\ph(N_{R}\times\{0\})\mid\sigma\in\Sigma\}$ is a fan. But $\Delta$ is the common refinement of $\Sigma$ and the generalized fan $\{\im(\ph),\ph(N_{R}\times\{0\})\}$, and so is a generalized fan containing the zero cone, i.e., a fan.

In the cases where this notion will be of use to us $W$ will be of the form $N_{R}\times R^k$ and $\Sigma$ will be a rational fan. 
In order to consider rational cones in $N_{R}\times R^k$ we use the pairing $\pairing\fncolon(M_\Q\times\Q^k)\times(N_R\times R^k)\to R$ given as $\angbra{(u,f),(w,\xi)}=\angbra{u,w}+f\cdot \xi$ where $f\cdot \xi$ denotes the usual dot product on $R^k$. We also use the dot product itself as a pairing $\pairing\fncolon\Q^k\times R^k\to R$. 

The features of the linear map in the conclusion of Theorem \ref{thm:CombinatorialNoetherianReduction} are designed to match the hypotheses of the following lemma.

\begin{lemma}\label{lemma:PullbackOfRationalIsAdmissible}
Assume that $\Gamma\neq\{0\}$ and let $\iota\fncolon R_{\geq0}\to R^k$ be a linear map which sends $1$ to a nonzero point of $\Gamma^k$. For any rational fan $\Sigma$ in $N_R\times R^k$, the pullback $(\id_{N_{R}}\times\iota)^*(\Sigma)$ of $\Sigma$ along $\id_{N_{R}}\times\iota$ is a $\Gamma$-admissible fan in $N_{R}\times R_{\geq0}$.
\end{lemma}\begin{proof}
We already know that $(\id_{N_{R}}\times\iota)^*(\Sigma)$ is a fan, so we just have to check that if $\sigma$ is a rational cone in $N_{R}\times R^k$, $(\id_{N_{R}}\times\iota)^{-1}(\sigma)$ is a $\Gamma$-admissible cone. To see this, note that if we write $\sigma=\{(u_1,f_1),\ldots,(u_m,f_m)\}^\vee$ with $u_1,\ldots,u_m\in M$ and $f_1,\ldots,f_m\in\Z^k$ then $(\id_{N_{R}}\times\iota)^{-1}(\sigma)=\{(u_1,\iota^*(f_1)),\ldots,(u_m,\iota^*(f_m))\}^\vee$. The $\Gamma$-admissibility of $(\id_{N_{R}}\times\iota)^{-1}(\sigma)$ follows because for $f\in\Z^k$, $\iota^*(f)=\angbra{\iota^*(f),1}=\angbra{f,\iota(1)}\in\Gamma$. 
\end{proof}
%
%

\section{Existence of $\Gamma$-admissible completions}\label{sec:ExistenceOfAdmissibleCompletions}

In this section we prove Theorems \ref{thm:AdmissibleCompletionsExist} and \ref{thm:CombinatorialNoetherianReduction}. Throughout this section $\Gamma$ will be a nontrivial additive subgroup of $\R$. We start by showing that Theorem \ref{thm:AdmissibleCompletionsExist} follows from Theorem \ref{thm:CombinatorialNoetherianReduction}.

\begin{prop}
Theorem \ref{thm:CombinatorialNoetherianReduction} implies Theorem \ref{thm:AdmissibleCompletionsExist}.
\end{prop}
\begin{proof}
Theorem \ref{thm:CombinatorialNoetherianReduction} tells us that there are a positive integer $k$, a rational fan $\widetilde{\Sigma}$ in $N_{\R}\times\R^k$, and a linear map $\iota\fncolon\R_{\geq0}\to\R^k$ which sends $1$ to a nonzero point of $\Gamma^k$ such that $\Sigma=(\id_{N_\R}\times\iota)^*(\widetilde{\Sigma})$. 
By \cite[page 18]{Oda} there is a complete rational fan $\widetilde{\Sigma}'$ in $N_{\R}\times\R^k$ containing $\widetilde{\Sigma}$. 
Let $\nbar{\Sigma}:=(\id_{N_\R}\times\iota)^*(\widetilde{\Sigma}')$. Since $\iota(1)\in\Gamma^k\sdrop\{0\}$ and $\nbar{\Sigma}$ is the pullback of a rational fan along $\id_{N_\R}\times\iota$, Lemma \ref{lemma:PullbackOfRationalIsAdmissible} gives us that $\nbar{\Sigma}$ is a $\Gamma$-admissible fan in $N_{\R}\times\R_{\geq0}$. Because $\widetilde{\Sigma}'$ contains $\widetilde{\Sigma}$ and $(\id_{N_\R}\times\iota)^*(\widetilde{\Sigma})=\Sigma$, $\nbar{\Sigma}=(\id_{N_\R}\times\iota)^*(\widetilde{\Sigma}')$ contains $\Sigma$. Finally, $|\nbar{\Sigma}|=(\id_{N_\R}\times\iota)^{-1}(|\widetilde{\Sigma}'|)=(\id_{N_\R}\times\iota)^{-1}(N_{\R}\times\R^k)=N_{\R}\times\R_{\geq0}$, so $\nbar{\Sigma}$ is complete.
\end{proof}

Thus our focus in the remainder of this section will be on proving Theorem \ref{thm:CombinatorialNoetherianReduction}. 
Our first step will be to reduce the proof to the following special case.

\begin{prop}\label{prop:ExtendFan}
Let $\Gamma$ be nontrivial subgroup of $\R$ such that $\Q\Gamma$ is finite dimensional as a $\Q$-vector space. There are a positive integer $k$ and a linear map $\iota\fncolon\R_{\geq0}\to\R^k$ sending $1$ to a nonzero point of $\Gamma^k$ which satisfy the following condition. For any $\Gamma$-admissible fan $\Sigma$ in $N_{\R}\times\R_{\geq0}$ such that no maximal cone of $\Sigma$ is contained in $N_{\R}\times\{0\}$ there is a rational fan $\widetilde{\Sigma}$ in $N_{\R}\times\R^k$ with $\Sigma=(\id_{N_{\R}}\times\iota)^*(\widetilde{\Sigma})$.
\end{prop}

In order to deduce Theorem \ref{thm:CombinatorialNoetherianReduction} from Proposition \ref{prop:ExtendFan} we need the following lemma.

\begin{lemma}\label{lemma:ExtendPastHeightZero}
Let $\Sigma$ be a $\Gamma$-admissible fan in $N_\R\times\R_{\geq0}$ and let $\sigma\in\Sigma$ be a maximal cone contained in $N_{\R}\times\{0\}$. Then $\sigma$ is a face of of a $\Gamma$-admissible cone $\sigma'$ not contained in $N_{\R}\times\{0\}$ such that $\Face(\sigma')\cup\Sigma$ is a fan.
\end{lemma}\begin{proof}
Let $\tilde{\sigma}$ be the rational cone in $N_{\R}$ with $\sigma=\tilde{\sigma}\times\{0\}$. If $L\subset M$ is a finite set such that $\tilde{\sigma}=L^\vee$ then $\tilde{\sigma}\times\R_{\geq0}=(L\times\{0\})^\vee$ so $\tilde{\sigma}\times\R_{\geq0}$ is a $\Gamma$-admissible cone. Fix $w_0\in\RelInt(\tilde{\sigma})$.

Part (\ref{lemmaPart:SeparationOfCones}) of Lemma \ref{lemma:AdmissibleSeparationLemma} gives us that for each $\tau\in\Sigma_{\max}\sdrop\{\sigma\}$ we can find $y_\tau\in M\times\Gamma$ such that $\sigma\subset y_{\tau}^\vee$, $\tau\subset(-y_\tau)^\vee$, and $\sigma\cap\tau=\sigma\cap y_\tau^\perp=\tau\cap y_\tau^\perp$. In particular we have $(w_0,0)\in\relint\sigma\subset\sigma\sdrop(\sigma\cap\tau)$ so $\angbra{y_\tau,(w_0,0)}>0$. So there is some $\eps_\tau>0$ in $\R$ such that $\{w_0\}\times[0,\eps_\tau]\subset y_\tau^\vee$.

Now let $\sigma':=(\tilde{\sigma}\times\R_{\geq0})\cap\{y_\tau\mid\tau\in\Sigma_{\max}\sdrop\{\sigma\}\}^\vee$. 
Because $\sigma$ is a face of $\tilde{\sigma}\times\R_{\geq0}$ and $\sigma\subset y_\tau^\vee$ for all $\tau\in\Sigma_{\max}\sdrop\{\sigma\}$, we see that $\sigma$ is a face of $\sigma'$. Thus in order to conclude that $\Face(\sigma')\cup\Sigma$ is a fan it suffices to show that $\sigma'\cap|\Sigma|\subset\sigma$. This inclusion follows from writing $|\Sigma|=\sigma\cup\dcup_{\tau\in\Sigma_{\max}\sdrop\{\sigma\}}\tau$ because we know that $\sigma'\cap\tau\subset y_\tau^\vee\cap\tau\subset\sigma$ for all $\tau\in\Sigma_{\max}\sdrop\{\sigma\}$. Finally, $\sigma'$ is not contained in $N_{\R}\times\{0\}$ because it contains the point $(w_0,\eps)$ where $\eps=\min\{\eps_\tau\mid\tau\in\Sigma_{\max}\sdrop\{\sigma\}\}$.
\end{proof}

\begin{lemma}
Theorem \ref{thm:CombinatorialNoetherianReduction} follows from Proposition \ref{prop:ExtendFan}.
\end{lemma}
\begin{proof}
Because $\Sigma$ is a finite collection of cones, each of which can be given as the locus in $N_{\R}\times\R_{\geq0}$ where a finite set of elements of $M\times\Gamma$ is nonnegative, $\Sigma$ is in fact $\Gamma'$-admissible for some finitely generated subgroup $\Gamma'$ of $\Gamma$. Fix $k$ and $\iota$ as in Proposition \ref{prop:ExtendFan}.

By recursively applying Lemma \ref{lemma:ExtendPastHeightZero} to any maximal cones of $\Sigma$ contained in $N_{\R}\times\{0\}$ we obtain a $\Gamma'$-admissible extension $\Sigma'$ of $\Sigma$ not having any maximal cones contained in $N_{\R}\times\{0\}$.

Proposition \ref{prop:ExtendFan} tells us that there is a rational fan $\widetilde{\Sigma}'$ in $N_{\R}\times\R^k$ such that $(\id_{N_{\R}}\times\iota)^*(\widetilde{\Sigma}')=\Sigma'$. 
Let $\widetilde{\Sigma}:=\{\sigma\in\widetilde{\Sigma}'\mid(\id_{N_{\R}}\times\iota)^{-1}(\sigma)\in\Sigma\}$. 
Suppose $\sigma\in\widetilde{\Sigma}$ and $\tau$ is a face of $\sigma$, so there is some $(u,f)\in M_{\Q}\times\Q^k$ such that $\sigma\subset (u,f)^\vee$ and $\tau=\sigma\cap (u,f)^\perp$. 
Thus the cone $(\id_{N_{\R}}\times\iota)^{-1}(\sigma)$ is contained in $(u,\iota^*(f))^\vee$ and so 
\begin{align*}
(\id_{N_{\R}}\times\iota)^{-1}(\tau)&=(\id_{N_{\R}}\times\iota)^{-1}(\sigma\cap (u,f)^\perp)\\
&=(\id_{N_{\R}}\times\iota)^{-1}(\sigma)\cap(u,\iota^*(f))^\perp
\end{align*}
 is a face of $(\id_{N_\R}\times\iota)^{-1}(\sigma)$. Since $(\id_{N_\R}\times\iota)^{-1}(\sigma)$ is in the fan $\Sigma$, this gives us that $(\id_{N_{\R}}\times\iota)^{-1}(\tau)\in\Sigma$ as well, and so $\tau\in\widetilde{\Sigma}$. Hence $\widetilde{\Sigma}$ is a subfan of $\widetilde{\Sigma}'$ such that $(\id_{N_\R}\times\iota)^*(\widetilde{\Sigma})=\Sigma$.

For the final claim, note that the only way that we used that $\Gamma'$ was finitely generated was to apply Proposition \ref{prop:ExtendFan}. Thus if $\Q\Gamma$ has finite dimension as a $\Q$-vector space then we can set $\Gamma'=\Gamma$ and then proceed as above.
\end{proof}

As we have reduced the proofs of Theorems \ref{thm:AdmissibleCompletionsExist} and \ref{thm:CombinatorialNoetherianReduction} to showing Proposition \ref{prop:ExtendFan}, for the remainder of this section we fix $\Gamma$ such that $\Q\Gamma$ is finite dimensional as a $\Q$-vector space. 
Fix $\gamma_1,\ldots,\gamma_k\in\Gamma$ which form a $\Q$-basis for $\Q\Gamma$. Let $\iota\fncolon\R_{\geq0}\to\R^k$ be the map $t\mapsto t\vgamma$ where $\vgamma=(\gamma_1,\ldots,\gamma_k)\in\R^k$.

We now outline the proof of Proposition \ref{prop:ExtendFan}. Suppose we are given a $\Gamma$-admissible fan $\Sigma$ in $N_{\R}\times\R_{\geq0}$ whose maximal cones all meet $N_{\R}\times\{1\}$ and we wish to show that there is a rational fan $\widetilde{\Sigma}$ such that $\Sigma=(\id_{N_\R}\times\iota)^*(\widetilde{\Sigma})$. By Lemma \ref{lemma:ComplexFromBigPolyhedra} we only need to construct the maximal cones of $\widetilde{\Sigma}$ and so can restrict our attention to the maximal cones of $\Sigma$. 
We will see that it is straightforward to find, for each maximal cone $\sigma$ of $\Sigma$, a rational cone $\widetilde{\sigma}$ in $N_{\R}\times\R^k$ such that $\sigma=(\id_{N_{\R}}\times\iota)^{-1}(\widetilde{\sigma})$. It will take significantly more work to show that there is a way to construct these $\widetilde{\sigma}$s so that for $\sigma_1,\sigma_2\in\Sigma_{\max}$, $\widetilde{\sigma}_1\cap\widetilde{\sigma}_2$ is a face of both $\widetilde{\sigma}_1$ and $\widetilde{\sigma}_2$.

Consider a $\Gamma$-admissible cone $\sigma$ in $N_{\R}\times\R_{\geq0}$; we want to find a rational cone $\widetilde{\sigma}$ in $N_{\R}\times\R^k$ with $\sigma=(\id_{N_{\R}}\times\iota)^{-1}(\widetilde{\sigma})$. Suppose we are also given a presentation of $\sigma$ as $\sigma=L^\vee$ for a finite set $L\subset M_{\Q}\times\Gamma$. One way to obtain a $\widetilde{\sigma}$ as above is to find, for each $y\in L$, some $\widetilde{y}$ in $M_{\Q}\times\Q^k$ such that $(\id_{N_{\R}}\times\iota)^{-1}(\widetilde{y}^\vee)=y^\vee$, and then to let $\widetilde{\sigma}=\{\widetilde{y}\mid y\in L\}^\vee$. This approach motivates us to make the following definitions.

For any $y\in M_{\Q}\times\Gamma$ write $y=\left(u,\sum_{\ell=1}^k y_\ell \gamma_\ell\right)$ with $y_{\ell}\in\Q$ and define $\widetilde{y}\in M_{\Q}\times \Q^k$ to be $\widetilde{y}:=(u,(y_1,\ldots,y_k))$. Further, for $L\subset M_{\Q}\times\Gamma$ we let $\widetilde{L}:=\{\widetilde{y}\mid y\in L\}$. The proof of Lemma \ref{lemma:TildeExtendsCones} shows that, for any such $L$, $(\id_{N_{\R}}\times\iota)^{-1}(\widetilde{L}^\vee)=L^\vee$. On the other hand, even if $L_1$ and $L_2$ are such that $L_1^\vee$ and $L_2^\vee$ intersect in a common face, the same need not hold for $\widetilde{L}_1$ and $\widetilde{L}_2$. Because of this, we will use a modification of the above construction.

Let $\pi\fncolon N_{\R}\times\R^k\to\R^k$ be the projection onto the second factor. For any $L\subset M_{\Q}\times\Gamma$ and any rational cone $B\subset\R^k$ we consider the cone $\posloc_{B}(\widetilde{L}):=\pi^{-1}(B)\cap\widetilde{L}^\vee$. 
We will work towards showing that for appropriate $L_1,L_2\subset M_{\Q}\times\Gamma$ we can choose $B$ so that $\posloc_B(\widetilde{L}_1)$ and $\posloc_B(\widetilde{L}_2)$ intersect in a common face. First, we verify that for suitable $B$, $(\id_{N_{\R}}\times\iota)^{-1}(\posloc_B(\widetilde{L}))=L^\vee$.

\begin{lemma}\label{lemma:TildeExtendsCones}
Let $B\subset\R^k$ be a cone containing $\vgamma$. For any $L\subset M_{\Q}\times\Gamma$ we have $(\id_{N_\R}\times\iota)^{-1}(\posloc_B(\widetilde{L}))=L^\vee$. In particular, $\id_{N_\R}\times\iota$ maps $L^\vee\cap(N_{\R}\times\{0_{\R}\})$ onto $\posloc_B(\widetilde{L})\cap(N_{\R}\times\{0_{\R^k}\})$.
\end{lemma}\begin{proof}
For $y\in M_{\Q}\times\Gamma$ and $(w,t)\in N_{\R}\times\R_{\geq0}$ we have $\angbra{\widetilde{y},(w,\iota(t))}=\angbra{y,(w,t)}$. Thus $(\id_{N_{\R}}\times\iota)^{-1}(\widetilde{y}^\vee)=y^\vee$. Since the image of $\id_{N_{\R}}\times\iota$ is contained in $\pi^{-1}(B)$ this gives 
\begin{align*}
(\id_{N_\R}\times\iota)^{-1}(\posloc_B(\widetilde{L}))
&=(\id_{N_\R}\times\iota)^{-1}\left(\pi^{-1}(B)\cap\dcap_{y\in L}\widetilde{y}^\vee\right)\\
&=\dcap_{y\in L}(\id_{N_\R}\times\iota)^{-1}(\widetilde{y}^\vee)=\dcap_{y\in L}y^\vee\\
&=L^\vee.
\end{align*}

The final statement follows because $\id_{N_\R}\times\iota$ maps $N_{\R}\times\{0_{\R}\}$ onto $N_{\R}\times\{0_{\R^k}\}$.
\end{proof}

Within the current framework, we will use Proposition \ref{prop:FarkasPlus} in the form of the following lemma.

\begin{lemma}\label{lemma:ExtendIncl}
Say $L\subset M_{\Q}\times\Gamma$ is a finite set such that $L^\vee$ is not contained in $N_{\R}\times\{0_{\R}\}$. Suppose also that $y_0\in M_{\Q}\times\Gamma$ is such that $L^\vee\subset y_0^\vee$. Then there is $f\in\Q^k$ with $\vgamma\in f^\vee$ such that if $B\subset f^\vee$ then $\posloc_B(\widetilde{L})\subset \widetilde{y_0}^\vee$.
\end{lemma}\begin{proof}
Write $L=\{y_1,\ldots,y_m\}$ and for $i=0,1,\ldots,m$ write $y_i=(u_i,\alpha_i)$ with $u_i\in M_{\Q}$ and $\alpha_i\in\Gamma$. Setting $P:=\{w\in N_{\R}\mid\angbra{u_i,w}\geq -\alpha_i\ (\forall i=1,\ldots,m)\}$ we have $L^\vee\cap(N_{\R}\times\{1\})=P\times\{1\}$. So $P$ is nonempty and $\angbra{u_0,w}\geq -\alpha_0$ is valid for all $w\in P$. Thus Proposition \ref{prop:FarkasPlus} gives us that there are $c_1,\ldots,c_m\in\Q_{\geq0}$ such that $\sum_{i=1}^m c_iu_i=u_0$ and $\delta:=\sum_{i=1}^m c_i\alpha_i\leq \alpha_0$. In particular, $\sum_{i=1}^m c_iy_i=(u_0,\delta)$. Write $\alpha_0-\delta=\sum_{\ell=1}^k f_{\ell}\gamma_{\ell}$ and let $f=(f_1,\ldots,f_k)\in\Q^k$, so $\widetilde{y}_0=\sum_{i=1}^mc_i\widetilde{y}_i+(0,f)$. Now if $B\subset f^\vee$ then $\posloc_B(\widetilde{L})\subset(0,f)^\vee$, so if $x\in \posloc_B(\widetilde{L})$ then $\angbra{\widetilde{y}_0,x}=\sum_{i=1}^m c_i\angbra{\widetilde{y}_i,x}+\angbra{(0,f),x}\geq0$.
\end{proof}

\begin{proof}[Proof of Proposition \ref{prop:ExtendFan}]
Let $\Sigma$ be a $\Gamma$-admissible fan in $N_{\R}\times\R_{\geq0}$ such that no maximal cone of $\Sigma$ is contained in $N_{\R}\times\{0\}$. For each $\sigma\in\Sigma_{\max}$ fix a finite $L_{\sigma}\subset M_{\Q}\times\Gamma$ such that $\sigma=L_{\sigma}^\vee$. We will show that there is a rational polyhedral cone $B\subset\R^k$ containing $\vgamma$ such that $\widetilde{\Sigma}:=\dcup_{\sigma\in\Sigma_{\max}}\Face(\posloc_B(\widetilde{L_{\sigma}}))$ is a rational fan with $\Sigma=(id_{N_{\R}}\times\iota)^*(\widetilde{\Sigma})$.

First we show that for $\sigma,\tau\in\Sigma_{\max}$ there is a rational cone $B_{\sigma,\tau}$ containing $\nbar{\gamma}$ such that if $B\subset B_{\sigma,\tau}$ then $\posloc_B(\widetilde{L_\sigma})$ and $\posloc_B(\widetilde{L_\tau})$ intersect in a common face. 
Because $\sigma$ and $\tau$ intersect in a common face, part (\ref{lemmaPart:SeparationOfCones}) of Lemma \ref{lemma:AdmissibleSeparationLemma} gives us that there is $y_0\in M_{\Q}\times\Gamma$ such that $\sigma\subset y_0^\vee$, $\tau\subset(-y_0)^\vee$, and $\sigma\cap(-y_0)^\vee=\sigma\cap\tau=\tau\cap y_0^\vee$. 
That is, we have the finitely many inclusions
\begin{align}
L_{\sigma}^\vee&\subset y_0^\vee,\label{eqn:Incl1}\\
L_{\tau}^\vee&\subset (-y_0)^\vee,\label{eqn:Incl2}\\
(L_{\sigma}\cup\{-y_0\})^\vee&\subset y^\vee\quad\text{ for }y\in L_{\sigma}\cup L_{\tau},\label{eqn:Incl3}\\
\intertext{and}
(L_{\tau}\cup\{y_0\})^\vee&\subset y^\vee\quad\text{ for }y\in L_{\sigma}\cup L_{\tau}\label{eqn:Incl4}.
\end{align}
In order to use these inclusions to conclude that $\posloc_B(\widetilde{L_\sigma})$ and $\posloc_B(\widetilde{L_\tau})$ intersect in a common face, we consider two cases as to whether $\sigma\cap\tau$ is contained in $N_{\R}\times\{0\}$ or not.

\emph{Case 1}: $\sigma\cap\tau$ is not contained in $N_{\R}\times\{0\}$. In this case $L_{\sigma}^\vee$, $L_{\tau}^\vee$, $(L_{\sigma}\cup\{-y_0\})^\vee=\sigma\cap\tau$, and $(L_{\tau}\cup\{y_0\})^\vee=\sigma\cap\tau$ are not contained in $N_{\R}\times\{0\}$, so we can apply Lemma \ref{lemma:ExtendIncl} to the inclusions in (\ref{eqn:Incl1}) - (\ref{eqn:Incl4}). We conclude that there is a finite set $F_{\sigma,\tau}\subset\Q^k$ with $\vgamma\in B_{\sigma,\tau}:=F_{\sigma,\tau}^\vee$ such that if $B\subset B_{\sigma,\tau}$ then 
\begin{align*}
\posloc_B(\widetilde{L_{\sigma}})&\subset \widetilde{y_0}^\vee, \\
\posloc_B(\widetilde{L_{\tau}})&\subset(-\widetilde{y_0})^\vee, \\
\posloc_B(\widetilde{L_{\sigma}})\cap(-\widetilde{y_0})^\vee&\subset\posloc_B(\widetilde{L_{\sigma}})\cap\posloc_B(\widetilde{L_{\tau}}), \\
\intertext{and}
\posloc_B(\widetilde{L_{\tau}})\cap \widetilde{y_0}^\vee&\subset\posloc_B(\widetilde{L_{\sigma}})\cap\posloc_B(\widetilde{L_{\tau}}).
\end{align*}
So $B_{\sigma,\tau}$ is as desired.

\emph{Case 2}: $\sigma\cap\tau$ is contained in $N_{\R}\times\{0\}$. Write $y_0=(u_0,\alpha_0)$ with $u_0\in M_{\Q}$ and $\alpha_0\in\Gamma$. Let $P_{\sigma}:=\{w\in N_{\R}\mid (w,1)\in\sigma\}$ and $P_{\tau}:=\{w\in N_{\R}\mid (w,1)\in\tau\}$, which are $\Gamma$-rational polyhedra. Further, $\angbra{u_0,w}>-\alpha_0$ is valid for $w\in P_{\sigma}$ and $\angbra{u_0,w}<-\alpha_0$ is valid for $w\in P_{\tau}$. Thus there is a positive $\eps\in\Q\Gamma$ such that $\angbra{u_0,w}\geq-\alpha_0+\eps$ is valid for $w\in P_{\sigma}$ and $\angbra{u_0,w}\leq-\alpha_0-\eps$ is valid for $w\in P_{\tau}$. In terms of the cones $\sigma$ and $\tau$ this gives us that $\sigma\subset (y_0-(0,\eps))^\vee$ and $\tau\subset(-y_0-(0,\eps))^\vee$. Since $\sigma$ and $\tau$ are not contained in $N_{\R}\times\{0\}$, Lemma \ref{lemma:ExtendIncl} tells us that there are $g_1,g_2\in\Q^k$ such that if $B\subset\{g_1,g_2\}^\vee$ then 
\begin{align}
\posloc_B(\widetilde{L_{\sigma}})&\subset \left(\widetilde{y_0}-\widetilde{(0,\eps)}\right)^\vee\label{eqn:C2Incl1}\\
\intertext{and}
\posloc_B(\widetilde{L_{\tau}})&\subset\left(-\widetilde{y_0}-\widetilde{(0,\eps)}\right)^\vee\label{eqn:C2Incl2}.
\end{align}

Write $\eps=\sum_{\ell=1}^k h_\ell\gamma_\ell$ with $h_\ell\in\Q$ and let $h=(h_1,\ldots,h_k)$ so $\widetilde{(0,\eps)}=(0,h)$. 
So for $B\subset\{g_1,g_2,h\}^\vee$ the inclusions (\ref{eqn:C2Incl1}) and (\ref{eqn:C2Incl2}) give us that also $\posloc_B(\widetilde{L_{\sigma}})\subset \widetilde{y_0}^\vee$ and $\posloc_B(\widetilde{L_{\tau}})\subset(-\widetilde{y_0})^\vee$. Furthermore, for $b\in B\sdrop h^\perp$ we have that $\posloc_B(\widetilde{L_{\sigma}})\cap\pi^{-1}(b)\subset \widetilde{y_0}^\vee\sdrop\widetilde{y_0}^\perp$ and $\posloc_B(\widetilde{L_{\tau}})\cap\pi^{-1}(b)\subset(-\widetilde{y_0})^\vee\sdrop\widetilde{y_0}^\perp$, so $\posloc_B(\widetilde{L_{\sigma}})\cap\posloc_B(\widetilde{L_{\tau}})\cap\pi^{-1}(b)=\emptyset$. Thus $\posloc_B(\widetilde{L_{\sigma}})\cap\posloc_B(\widetilde{L_{\tau}})\subset \pi^{-1}(B\cap h^\perp)$.

Because $\eps$ is positive, $\nbar{\gamma}$ is in $h^\vee\sdrop h^\perp$, so by Lemma \ref{lemma:ThinRationalConeAroundPoint} there is a rational cone $B'\subset h^\vee$ containing $\nbar{\gamma}$ with $B'\cap h^\perp=\{0_{\R^k}\}$. 
So for $B\subset B_{\sigma,\tau}:=\{g_1,g_2,h\}^\vee\cap B'$ we get that 
\begin{equation}
\posloc_B(\widetilde{L_{\sigma}})\cap\posloc_B(\widetilde{L_{\tau}})\subset \pi^{-1}(B\cap h^\perp)=N_{\R}\times\{0_{\R^k}\}.\label{eqn:C2TildeIntersectAtHeightZero}
\end{equation}
Because $B'$, and so also $B$, is a strongly convex cone we have that 
\begin{align}
\posloc_B(\widetilde{L_{\sigma}})\cap (N_{\R}\times\{0_{\R^k}\})&\leq \posloc_B(\widetilde{L_{\sigma}})\label{eqn:C2HeightZeroFace1}\\
\intertext{and}
\posloc_B(\widetilde{L_{\tau}})\cap (N_{\R}\times\{0_{\R^k}\})&\leq\posloc_B(\widetilde{L_{\tau}}).\label{eqn:C2HeightZeroFace2}
\end{align}
By Lemma \ref{lemma:TildeExtendsCones} restrictions of $\id_{N_{\R}}\times\iota$ give bijections 
\begin{align*}
\sigma\cap(N_{\R}\times\{0_{\R}\})&\to\posloc_B(\widetilde{L_{\sigma}})\cap (N_{\R}\times\{0_{\R^k}\}),\\
\tau\cap(N_{\R}\times\{0_{\R}\})&\to\posloc_B(\widetilde{L_{\tau}})\cap (N_{\R}\times\{0_{\R^k}\}),\\
\intertext{and}
\sigma\cap\tau\cap(N_{\R}\times\{0_{\R}\})&\to\posloc_B(\widetilde{L_{\sigma}})\cap\posloc_B(\widetilde{L_{\tau}})\cap (N_{\R}\times\{0_{\R^k}\}).
\end{align*}
Together with the fact that $\sigma\cap(N_{\R}\times\{0_{\R}\})$ and $\tau\cap(N_{\R}\times\{0_{\R}\})$ are both faces of the fan $\Sigma$ this gives us that $\posloc_B(\widetilde{L_{\sigma}})\cap\posloc_B(\widetilde{L_{\tau}})\cap (N_{\R}\times\{0_{\R^k}\})$ is a common face of $\posloc_B(\widetilde{L_{\sigma}})\cap (N_{\R}\times\{0_{\R^k}\})$ and $\posloc_B(\widetilde{L_{\tau}})\cap (N_{\R}\times\{0_{\R^k}\})$. 
Combining this with (\ref{eqn:C2TildeIntersectAtHeightZero}), (\ref{eqn:C2HeightZeroFace1}), and (\ref{eqn:C2HeightZeroFace2}) we see that $\posloc_B(\widetilde{L_{\sigma}})\cap\posloc_B(\widetilde{L_{\tau}})=\posloc_B(\widetilde{L_{\sigma}})\cap\posloc_B(\widetilde{L_{\tau}})\cap (N_{\R}\times\{0_{\R^k}\})$ is a face of both $\posloc_B(\widetilde{L_{\sigma}})$ and $\posloc_B(\widetilde{L_{\tau}})$. 
This shows that $B_{\sigma,\tau}$ is as desired, finishing the proof of the original claim in case 2.

Now $B_{\Sigma}:=\dcap_{\sigma,\tau\in\Sigma_{\max}}B_{\sigma,\tau}$ is a rational polyhedral cone containing $\nbar{\gamma}$. Let $B$ be any strongly convex rational polyhedral cone containing $\nbar{\gamma}$ and contained in $B_{\Sigma}$. Then for any $\sigma,\tau\in\Sigma_{\max}$ we have that $\posloc_B(\widetilde{L_{\sigma}})$ and $\posloc_B(\widetilde{L_{\tau}})$ intersect in a common face, so by Lemma \ref{lemma:ComplexFromBigPolyhedra} $\widetilde{\Sigma}:=\dcup_{\sigma\in\Sigma_{\max}}\Face(\posloc_B(\widetilde{L_{\sigma}}))$ is a generalized fan. Note that each $\pi(\posloc_B(\widetilde{L_{\sigma}}))\subset B$, so because $B$ is strongly convex the lineality space of $\posloc_B(\widetilde{L_{\sigma}})$ is contained in $\pi^{-1}(0)=N_{\R}\times\{0_{\R^k}\}$. So the lineality space of $\posloc_B(\widetilde{L_{\sigma}})$ is also the lineality space of $\posloc_B(\widetilde{L_{\sigma}})\cap(N_{\R}\times\{0_{\R^k}\})=(\id_{N_{\R}}\times\iota)(\sigma\cap(N_{\R}\times\{0_{\R}\}))$. Because $\sigma$ is strongly convex this must be $0$, so $\widetilde{\Sigma}$ is a fan.

We know that $\Delta:=(\id_{N_{\R}}\times\iota)^*(\widetilde{\Sigma})$ is a fan in $N_{\R}\times\R_{\geq0}$. The maximal cones of $\Delta$ are among the cones $(\id_{N_{\R}}\times\iota)^{-1}(\tau)$ for $\tau$ a maximal cone of $\widetilde{\Sigma}$. Because every maximal cone of $\widetilde{\Sigma}$ is of the form $\posloc_B(\widetilde{L_\sigma})$ for some maximal cone $\sigma$ in $\Sigma$ and $\sigma=(\id_{N_{\R}}\times\iota)^{-1}(\posloc_B(\widetilde{L_\sigma}))$ for such a $\sigma$, we have that all maximal cones of $\Delta$ are maximal cones of $\Sigma$. In particular, $\Delta$ is a subfan of $\Sigma$. Also, each maximal cone $\sigma$ in $\Sigma$ can be written as $\sigma=(\id_{N_{\R}}\times\iota)^{-1}(\posloc_B(\widetilde{L_{\sigma}}))$ and so is also a cone of $\Delta$. Thus $\Delta$ is a subfan of $\Sigma$ which contains all of the maximal cones of $\Sigma$, so $\Sigma=\Delta$.
\end{proof}

%
%

\section{Completions of fans over arbitrary ordered fields}\label{sec:CompletionsOverOrderedFields}

We now consider the existence of completions of fans when $\R$ is replaced by an arbitrary ordered field $R$. We show that rational fans in $N_{R}$ admit rational completions whereas for some ordered fields $R$ there are $R$-admissible fans in $N_{R}\times R_{\geq0}$ with no $R$-admissible completion. This section is independent of the rest of the paper and, in particular, is not needed for the proof of Theorem \ref{thm:NormalTCompletionsExist} or for any of the other results in Section \ref{sec:CompletionsOfTToricVars}.

Our proof of the existence of rational completions of rational fans in $N_{R}$ uses elementary model theory to deduce the result from the corresponding result for rational fans in $N_{\R}$. We refer to \cite{MarkerModelTheoryBook} for unexplained language and notation from model theory.

Let $\calL=\{0,+,-,<\}$ and let $\ODAG$ be the $\calL$-theory of nontrivial ordered divisible abelian groups.

\begin{prop}\label{prop:RationalFansInHigherRank}
Let $R$ be an ordered field. Then every rational fan in $N_{R}$ has a rational completion.
\end{prop}\begin{proof}
Fix an identification $N\cong\Z^n$ so that we can consider rational fans in $R^n$. Note that every rational cone in $R^n$ is $\emptyset$-definable in $\calL$. Moreover, if $\sigma_1,\ldots,\sigma_m$ are rational cones in $R^n$ then the property that $\sigma_i$ is the face of $\sigma_j$ cut out by $(\ell_1,\ldots,\ell_n)\in \Z^n$ is a first-order property, as is the property that the union $\bigcup_{i=1}^n \sigma_i$ is all of $R^n$. Since $R$ and $\R$ are both models of $\ODAG$, the existence of a rational completion of a fixed rational fan $\Sigma$ in $R^n$ follows from the existence of rational completions of rational fans in $\R^n$, \cite[page 18]{Oda}, and the fact that $\ODAG$ is a complete theory, \cite[Corollary 3.1.17]{MarkerModelTheoryBook}.
\end{proof}

In contrast to the case of rational fans, for some ordered fields $R$ there are $R$-admissible fans in $N_{R}\times R_{\geq0}$ which do not have any $R$-admissible completion. 
Before we begin the construction of such $R$-admissible fans we review some basic notions regarding the order and arithmetic of a general ordered field.

Let $R$ be an ordered field, so in particular, $\Q$ is a subfield of $R$. For non-negative elements $a,b\in R$ we say that $b$ is \emph{infinitesimal compared to $a$} and write $b\ll a$ if for all positive $q\in\Q$, $qb<a$. 
If this is not the case, i.e., there is some positive $q\in\Q$ such that $a\leq qb$ then we write $a\lesssim b$. 
Note that if $b\ll a$ and $c\lesssim b$ then $c\ll a$. 
The following lemma is well-known to those who work with arbitrary ordered fields; we include a quick proof in order to make our proof of Theorem \ref{thm:HigherRankBadFan} accessible to people who generally work only with polyhedra over $\R$.

\begin{lemma}\label{lemma:OrderedArithmetic}
Suppose $a,\delta\in R$ are positive, $\delta\ll a$, and $\theta\in R$ is such that $|\theta|\ll a$. Then $\delta\ll a+\theta$.
\end{lemma}\begin{proof}
If $\theta$ is non-negative then for any positive $q\in\Q$ we have $q\delta<a\leq a+\theta$, so $\delta\ll a+\theta$. 
If $\theta$ is negative then for any positive $q\in\Q$ we have $2q\delta<a$ and $-2\theta=2|\theta|<a$, so $q\delta-\theta<\frac{a}{2}+\frac{a}{2}=a$, and so $\delta\ll a+\theta$.
\end{proof}

The following lemma will allow us to restrict what $R$-admissible completions of a certain fan could look like, and so show that no such completion exists.

\begin{lemma}\label{lemma:ThisLineIsHorizontal}
Let $a,b\in R$ be positive and let $\eps,\gamma,\theta\in R$ be such that $0\leq\eps<\gamma$ and $\gamma,|\theta|\ll a,b$. Let $P$ be an $R$-rational polyhedron in $R^2$ which has the edge $\conv((-a,\eps),(b,\eps))$ as a face and suppose the point $(\theta,\gamma)$ is on the boundary of $P$. Then $P$ has a unique edge containing $(\theta,\gamma)$. This edge is bounded and its affine span is a line of slope $0$.
\end{lemma}\begin{proof}
Given an edge of $P$ containing $(\theta,\gamma)$ let $L$ be its affine span and let $m$ be the slope of $L$. Note that $P$ contains the points $(-a,\eps)$ and $(b,\eps)$ which are on two different sides of the line $\{(x,y)\in R^2\mid y=\theta\}$, so $m$ is a well defined rational number. So now $L=\{(x,y)\in R^2\mid y=m(x-\theta)+\gamma\}$, and $P$ is contained in one of the closed half-spaces of $R^2$ defined by $L$. Because $(\theta,\eps)$ is in $P$ and $m(\theta-\theta)+\gamma>\eps$, we see that $P\subset\{(x,y)\in R^2\mid y\leq m(x-\theta)+\gamma\}$. Now from the fact that $(-a,\eps)$ is in $P$ we get that $\eps\leq m(-a-\theta)+\gamma$ so $m(a+\theta)\leq\gamma-\eps\leq\gamma$. Then because Lemma \ref{lemma:OrderedArithmetic} gives us that $\gamma\ll a+\theta$ we must have $m\leq 0$. Similarly, from having $(b,\eps)$ in $P$ we get $\eps\leq m(b-\theta)+\gamma$ so $-m(b-\theta)\leq\gamma-\eps\leq\gamma$. Then from Lemma \ref{lemma:OrderedArithmetic} we know $\gamma\ll b-\theta$ and so $-m\leq0$, i.e., $m\geq0$, and so $m=0$.

Since all edges of $P$ containing $(\theta,\gamma)$ have the same slope, there can only be one such edge. If this edge were unbounded, then $P$ would contain a ray of the form $(\theta,\gamma)+\{(tr,0)\mid t\in R_{\geq0}\}$ for some nonzero $r\in R$. Hence $\{(tr,0)\mid t\in R_{\geq0}\}$ would be contained in the recession cone $\rec P$ of $P$, which is impossible because $(0,\eps)$ is in $P$ but $(0,\eps)+\{(tr,0)\mid t\in R_{\geq0}\}$ would contain either $(-2a,\eps)$ or $(2b,\eps)$, depending on whether $r$ is positive or negative, but neither of these points is in $P$.
\end{proof}

We are now ready to give our example of an $R$-admissible fan in $R^2\times R_{\geq0}$ with no $R$-admissible completion. Recall that an ordered field is archimedean if for any positive $a,b\in R$ we have $a\lesssim b$. An ordered field is archimedean if and only if it admits an embedding into $\R$.

\begin{thm}\label{thm:HigherRankBadFan}
Let $R$ be a non-archimedean ordered field and pick positive $a,\delta\in R$ with $\delta\ll a$. Let $\Sigma$ be the $R$-admissible fan in $R^2\times R_{\geq0}$ whose maximal cones are $\{(x,y,t)\in R^2\times R_{\geq0}\mid x=0,y=\delta t\}$ and $\{(x,y,t)\in R^2\times R_{\geq0}\mid y=0,-at\leq x\leq at\}$. 
Then $\Sigma$ has no $R$-admissible completion.
\end{thm}\begin{proof}
Suppose that $\nbar{\Sigma}$ were an $R$-admissible completion of $\Sigma$, and let $\Pi$ be the projection to $R^2$ of the restriction of $\nbar{\Sigma}$ to $R^2\times\{1\}$, i.e., $\Pi$ is the set of nonempty polyhedra $P\subset R^2$ such that $P\times\{1\}=\sigma\cap(R^2\times\{1\})$ for some $\sigma\in\nbar{\Sigma}$. Then $\Pi$ is a complete $R$-rational polyhedral complex containing the vertex $(0,\delta)$ and the line segment between $(-a,0)$ and $(a,0)$. Let $L$ be the line segment between $(0,0)$ and $(0,\delta)$. The restriction of $\Pi$ to $L$ consists of edges $E_1,\ldots,E_m$ and vertices $(0,0)=v_0,v_1,\ldots,v_m=(0,\delta)$ such that for $i=1,\ldots,m$ $v_{i-1}$ and $v_i$ are faces of $E_i$. For each $i=0,\ldots,m$ write $v_i=(0,\gamma_i)$.

We now show by induction on $i\geq1$ that any $P\in\Pi$ such that $P\cap L=E_i$ has as a face the line segment between $(-a_i,\gamma_i)$ and $(b_i,\gamma_i)$ for some positive $a_i,b_i\in R$ with $\delta\ll a_i,b_i$. For the base case of $i=1$, consider such a $P$ and let $a_0=a$ and $b_0=a$ so we know that $\delta\ll a_0,b_0$. Since $F_0:=\conv((-a_0,\gamma_0),(b_0,\gamma_0))$ is in $\Pi$ and $P$ contains the point $v_0=(0,\gamma_0)$ which is in the relative interior of $F_0$, $F_0$ must be a face of $P$. We have that $v_1\in E_1=L\cap P$ so $v_1\in P$ and we claim that $v_1$ is in the boundary of $P$. To see this, we consider two cases as to whether $m=1$ or $m>1$. If $m=1$ then $v_1=(0,\delta)$ which we know is in $\Pi$, so because $P\in\Pi$ contains $v_1$, $v_1$ must be a vertex of $P$. For the case $m>1$, recall that any face of a finite intersection of polyhedra can be written as an intersection of faces of the original polyhedra. So because $v_1$ is a proper face of $E_1=L\cap P$ and $m>1$ gives us that $v_1$ is in the relative interior of $L$, $v_1$ must be contained in a proper face of $P$, i.e., $v_1$ is in the boundary of $P$. Also note that $\gamma_1\leq\delta\ll a_0,b_0$, so $\gamma_1\ll a_0,b_0$. So we can apply Lemma \ref{lemma:ThisLineIsHorizontal} to conclude that $P$ has a unique edge $F_1$ containing $(0,\gamma_1)$ which is bounded and whose affine span is a line of slope $0$. Hence we can write $F_1=\conv((-a_1,\gamma_1),(b_1,\gamma_1))$ for some non-negative $a_1,b_1\in R$. It remains only to show that $\delta\ll a_1,b_1$. Suppose that $a_1\lesssim \delta$. Then $a_1\lesssim\delta\ll a_0,b_0$ gives us that $|-a_1|=a_1\ll a_0,b_0$. Hence another application of Lemma \ref{lemma:ThisLineIsHorizontal} shows that $F_1$ is the unique edge of $P$ containing the vertex $(-a_1,\gamma_1)$. But $P$ is a 2-dimensional polyhedron as it is in $R^2$ and contains the three affinely independent points $(-a_0,\gamma_0)$, $(b_0,\gamma_0)$, and $(0,\gamma_1)$, and so every vertex of $P$ must be contained in two distinct edges of $P$. This contradiction shows that we must have $\delta\ll a_1$. The proof that $\delta\ll b_1$ is exactly analogous.

For the inductive step, note that by the inductive hypothesis $\Pi$ contains an edge $F_{i-1}:=\conv((-a_{i-1},\gamma_{i-1}),(b_{i-1},\gamma_{i-1}))$ for some positive $a_{i-1},b_{i-1}\in R$ with $\delta\ll a_{i-1},b_{i-1}$. The rest of the proof proceeds exactly as the proof of the base case upon replacing appropriate $1$s with $i$s and appropriate $0$s with $(i-1)$s.

In particular, $\Pi$ contains an edge $\conv((-a_m,\gamma_m),(b_m,\gamma_m))$ with $\gamma_m\ll a_m,b_m$. The point $(0,\gamma_m)=(0,\delta)$, which is in $\Pi$, is in the relative interior of this edge, contradicting the fact that $\Pi$ was a polyhedral complex. Hence there cannot be any $R$-admissible completion of $\Sigma$.
\end{proof}

%
%

\section{Completions of $\T$-toric varieties}\label{sec:CompletionsOfTToricVars}

In this section we consider completions of $\T$-toric varieties. We start by recalling, in Section \ref{subsec:BackgroundTToric}, the relevant background on $\T$-toric varieties from \cite{GublerGuideTrop} and \cite{GublerSoto}. In Section \ref{subsec:NormalCompletions} we use Theorem \ref{thm:AdmissibleCompletionsExist} to deduce Theorem \ref{thm:NormalTCompletionsExist}. In Section \ref{subsec:BadNormalization} we give examples of projective $\T$-toric varieties whose normalizations are not of finite type. Finally, in Section \ref{subsec:SemistableTToricVars} we use Example \ref{example:FanWOASimplicialCompletion} to show that there are semistable $\T$-toric varieties which do not have a semistable equivariant completion.

\subsection{Background on $\T$-toric varieties}\label{subsec:BackgroundTToric}
Throughout this section $K$ will be a field with a nontrivial valuation $v\fncolon K\to\R\cup\{\infty\}$, $K^\circ$ will be the valuation ring of $v$, and $\Gamma$ will be the value group of $v$. Let $\T$ be a split torus over $K^\circ$ with character lattice $M$ and cocharacter lattice $N$. 
A \emph{$\T$-toric scheme} is an integral scheme $\calX$ which is separated and flat over $K^\circ$ such that the generic fiber $\calX_\eta$ of $\calX$ contains $\T_{\eta}$ as a dense open subset and the action of $\T_\eta$ on itself by translation extends to an action of $\T$ on $\calX$. A \emph{$\T$-toric variety} is a $\T$-toric scheme which is of finite type over $K^\circ$.

We now recall the correspondence between $\Gamma$-admissible fans in $N_{\R}\times\R_{\geq0}$ and normal $\T$-toric varieties. Given a $\Gamma$-admissible cone $\sigma\subset N_{\R}\times\R_{\geq0}$ we consider the $K^\circ$-algebra
$$K[M]^\sigma:=\left\{\sum_{u\in M}\lambda_u\chi^u\in K[M]\;\middle|\;(\forall(w,t)\in\sigma)\angbra{u,w}+tv(\lambda_u)\geq0\right\}.$$
The affine normal $\T$-toric scheme corresponding to $\sigma$ is $\calU_\sigma:=\spec K[M]^\sigma$. 
If $\tau$ is a face of $\sigma$ then the inclusion $K[M]^\sigma\subset K[M]^\tau$ induces a $\T$-equivariant open immersion. 
If $\Sigma$ is a $\Gamma$-admissible fan in $N_{\R}\times\R_{\geq0}$ then gluing $\calU_\sigma$ for $\sigma\in\Sigma$ along these open immersions gives a normal $\T$-toric scheme $\calY(\Sigma)$. When we need to keep track of the valuation ring, we will denote this by $\calY(\Sigma,K^\circ)$. This is sometimes necessary because for some fans $\Sigma$ and valued extensions $L/K$, $\calY(\Sigma,L^\circ)$ does not coincide with the base change $\calY(\Sigma,K^\circ)_{L^\circ}$ of $\calY(\Sigma,K^\circ)$ to $L^\circ$; see \cite[Propositions 7.11 and 7.12]{GublerGuideTrop}.

If $\Gamma$ is discrete then Gordan's lemma gives us that $\sigma^\vee\cap(M\times\Gamma)$ is a finitely generated monoid, so $K[M]^\sigma$ is a finitely generated $K^\circ$-algebra and $\calU_\sigma$ is a $\T$-toric variety. Let $P_\sigma\subset N_{\R}$ be the projection to $N_{\R}$ of $\sigma\cap(N_\R\times\{1\})$, so $\sigma\cap(N_\R\times\{1\})=P_\sigma\times\{1\}$. 
If $\Gamma$ is not discrete then \cite[Proposition 6.9]{GublerGuideTrop} shows that $K[M]^\sigma$ is finitely generated as a $K^\circ$-algebra if and only if all of the vertices of $P_\sigma$ are in $N_{\Gamma}$. 
If $P_\sigma$ is empty then $\sigma$ is a rational cone in $N_{\R}\times\{0\}$ and $\calU_\sigma$ is the corresponding toric variety over $K$. If $P_\sigma$ is not empty, let $w_1,\ldots,w_m$ be the vertices of $P_\sigma$. For each $w_i$ fix a finite generating set $\{u_{ij}\}$ for $LC_{w_i}(P_\sigma)^\vee\cap M$ where $LC_{w_i}(P_\sigma)=\R_{\geq0}(P_\sigma-w_i)$ is the local cone of $P_\sigma$ at $w_i$. If the $w_i$ are all in $N_{\Gamma}$ we can pick $\lambda_{ij}\in K$ such that $v(\lambda_{ij})=-\angbra{u_{ij},w_i}$, and then $K[M]^\sigma$ is generated as a $K^\circ$-algebra by the terms $\lambda_{ij}\chi^{u_{ij}}$; see the last paragraph of the proof of \cite[Proposition 6.7]{GublerGuideTrop}.

Given a $\Gamma$-admissible fan $\Sigma$ in $N_{\R}\times\R_{\geq0}$, the \emph{restriction of $\Sigma$ to $N_{\R}\times\{1\}$} is the polyhedral complex $\Sigma|_{N_{\R}\times\{1\}}$ consisting polyhedra of the form $\sigma\cap N_{\R}\times\{1\}$ for $\sigma\in\Sigma$. With this language, we have the following bijection, first established in \cite[IV, section 3]{ToroidalEmbeddings} in the case where $\Gamma$ is discrete and in \cite[Theorem 3]{GublerSoto} in the case where $\Gamma$ is not discrete.

\begin{thm}
If $\Gamma$ is discrete then $\Sigma\mapsto\calY(\Sigma)$ gives a bijection from the set of $\Gamma$-admissible fans in $N_{\R}\times\R_{\geq0}$ to the set of isomorphism classes of normal $\T$-toric varieties. If $\Gamma$ is not discrete then $\Sigma\mapsto\calY(\Sigma)$ gives a bijection from the set of $\Gamma$-admissible fans $\Sigma$ in $N_{\R}\times\R_{\geq0}$ such that all vertices of $\Sigma|_{N_{\R}\times\{1\}}$ are in $N_{\Gamma}\times\{1\}$ to the set of isomorphism classes of normal $\T$-toric varieties.
\end{thm}

\begin{remark}
For any $\Gamma$-admissible fan $\Sigma$ all of the vertices of $\Sigma|_{N_{\R}\times\{1\}}$ are in $N_{\Q\Gamma}\times\{1\}$. In particular, if $\Gamma$ is divisible then $\Sigma\mapsto\calY(\Sigma)$ gives a bijection from the set of $\Gamma$-admissible fans in $N_{\R}\times\R_{\geq0}$ to the set of isomorphism classes of normal $\T$-toric varieties.
\end{remark}

\subsection{Existence of normal completions}\label{subsec:NormalCompletions}

Let $L/K$ be a finite field extension such that the valuation $v$ extends uniquely to a valuation $L\to\R\cup\{\infty\}$. 
Let $k$ and $\ell$ be the residue fields of $K$ and $L$, respectively. 
We refer to \cite[II, \S7]{Neukirch} for the definition of what it means for the extension $L/K$ to be \emph{totally ramified}. We will only need to know that if $[\ell:k]=1$ then $L/K$ is totally ramified, and this follows easily from the definition. 

Although the following proposition is elementary and likely known to experts, we were unable to easily find a reference to it in the literature, and so we include a proof.

\begin{prop}\label{prop:EnoughTotallyRamifiedExtensions}
Let $K$ be a field with a valuation $v\fncolon K\to\R\cup\{\infty\}$. Let $\Gamma=v(K^\times)$ be the value group of $K$, and suppose $\Gamma'$ is an additive subgroup of $\R$ which contains $\Gamma$ such that $(\Gamma':\Gamma)$ is finite. Then there is a finite separable totally ramified extension $L/K$ such that the value group of $L$ is $\Gamma'$.
\end{prop}\begin{proof}
We will prove the slightly stronger statement that under the hypotheses above there is a finite separable field extension $L/K$ such that $v$ extends uniquely to $L$, the value group of $L$ is $\Gamma'$, and the corresponding extension on residue fields is the trivial extension. We initially show how to find such an extension which is not necessarily separable, and then show how to modify the construction to make $L$ separable.

We first consider the case where $\Gamma'$ is of the form $\Gamma'=\Gamma+\Z\alpha$. 
Note that $n:=(\Gamma':\Gamma)$ is the least positive integer such that $n\alpha$ is in $\Gamma$, and pick some $a\in K$ with $v(a)=n\alpha$. Any root of $f(x):=x^n-a$ in any valued extension of $K$ has valuation $\alpha$ so we find that $f(x)$ is irreducible in $K[x]$. Let $L=K(c)$ where $c$ is a root of $f$.

Since the completion $\widehat{K}$ of $K$ also has value group $\Gamma$, we also find that $f(x)$ is irreducible in $\widehat{K}[x]$. By \cite[II, Proposition 8.2]{Neukirch}, this implies that $v$ extends uniquely to a valuation on $L$. Let $\Gamma_L$ be the value group of $L$ and $\ell$ the residue field of $L$. Because $c$ has valuation $\alpha$ we have $\Gamma_L\geq\Gamma+\Z\alpha=\Gamma'$. We also know from \cite[II, Proposition 6.8]{Neukirch} that $[L:K]\geq(\Gamma_L:\Gamma)[\ell:k]$ where $k$ is the residue field of $K$. Thus
$$n=[L:K]\geq(\Gamma_L:\Gamma)[\ell:k]=(\Gamma_L:\Gamma')(\Gamma':\Gamma)[\ell:k]
=n(\Gamma_L:\Gamma')[\ell:k]$$
and so $(\Gamma_L:\Gamma')=1=[\ell:k]$, i.e., $\Gamma_L=\Gamma'$ and $\ell=k$.

We now prove the result for an arbitrary $\Gamma'$ by induction on $(\Gamma':\Gamma)$. The base case of $(\Gamma':\Gamma)=1$ is trivial. For the inductive step, suppose $(\Gamma':\Gamma)>1$ and pick some $\alpha\in \Gamma'\sdrop\Gamma$. By the previous case, there is a finite field extension $L_1/K$ such that $v$ extends uniquely to a valuation $v_1$ on $L_1$, the value group of $L_1$ is $\Gamma+\Z\alpha$ and the corresponding extension of residue fields is trivial. Since $\alpha\notin\Gamma$ we have $(\Gamma':\Gamma)>(\Gamma':\Gamma+\Z\alpha)$, so by the inductive hypothesis there is a finite field extension $L/L_1$ such that $v_1$ extends uniquely to a valuation on $L$, the value group of $L$ is $\Gamma'$, and the corresponding extension of residue fields is trivial. Then $L/K$ is an extension of $K$ with the desired properties.

Finally, to see that we can find such an extension $L/K$ which is also separable, note that in the first case the only feature of the polynomial $f(x)=x^n-a$ which we needed is that any root of $f(x)$ in a valued extension of $K$ has to have value $\alpha$. 
By \cite[II, Proposition 6.3]{Neukirch}, the same applies if we use $f(x)=x^n-bx-a$ for any $b\in K$ with $v(b)\geq(n-1)\alpha$. Thus, picking a nonzero $b\in K$ with $v(b)\geq(n-1)\alpha$, the proof above goes through when we replace $f(x)=x^n-a$ with $f(x)=x^n-bx-a$. 
In particular, we see that $f(x)$ is irreducible, and so, because $f(x)$ contains a linear term, $f(x)$ is separable. It follows that the field extensions constructed above are all separable.
\end{proof}

\begin{remark}
If we allow infinite separable totally ramified extensions, then the same result is true if we relax the condition that $(\Gamma':\Gamma)$ is finite to require only that $\Gamma\leq\Gamma'\leq\Q\Gamma$. The only difference in the proof is that the induction needs to be replaced by a Zorn's lemma argument or a transfinite induction argument.
\end{remark}

\begin{proof}[Proof of Theorem \ref{thm:NormalTCompletionsExist}]
Let $\Sigma$ be the $\Gamma$-admissible fan in $N_{\R}\times\R_{\geq0}$ such that $\calX\cong\calY(\Sigma)$. By Theorem \ref{thm:AdmissibleCompletionsExist} there is a complete $\Gamma$-admissible fan $\nbar{\Sigma}$ containing $\Sigma$. If $\Gamma$ is discrete or divisible then $\nbar{\calX}:=\calY(\nbar{\Sigma})$ is a normal $\T$-toric variety and by \cite[Proposition 11.8]{GublerGuideTrop} $\nbar{\calX}$ is proper over $K^\circ$. By the construction of the toric scheme corresponding to a $\Gamma$-admissible fan, the inclusion $\Sigma\subset\nbar{\Sigma}$ induces a $\bbT$-equivariant open immersion $\calX\into\nbar{\calX}$. So we assume from now on that $\Gamma$ is neither discrete nor divisible.

Let $\Sigma_1$ and $\nbar{\Sigma}_1$ be the restrictions of $\Sigma$ and $\nbar{\Sigma}$, respectively, to $N_{\R}\times\{1\}$. All of the vertices of $\Sigma_1$ are in $N_\Gamma\times\{1\}$ and the $\Gamma$-admissibility of $\nbar{\Sigma}$ gives us that all of the vertices of $\nbar{\Sigma}_1$ are in $N_{\Q\Gamma}\times\{1\}$. Since $\nbar{\Sigma}_1$ has finitely many vertices there is a subgroup $\Gamma'$ of $\Q\Gamma$ containing $\Gamma$ such that $(\Gamma':\Gamma)$ is finite and all of the vertices of $\nbar{\Sigma}_1$ are in $N_{\Gamma'}\times\{1\}$. 
By Proposition \ref{prop:EnoughTotallyRamifiedExtensions} there is a finite separable totally ramified extension $L/K$ such that the value group of $L$ is $\Gamma'$.

Since the vertices of $\Sigma_1$ are in $N_{\Gamma}\times\{1\}$, \cite[Proposition 7.12]{GublerGuideTrop} gives us that the base change $\calX_{L^\circ}$ of $\calX$ to the valuation ring $L^\circ$ of $L$ is isomorphic to $\calY(\Sigma,L^\circ)$ as $\T_{L^\circ}$-toric varieties. Let $\nbar{\calX}:=\calY(\nbar{\Sigma},L^\circ)$ which, by \cite[Proposition 11.8]{GublerGuideTrop}, is proper over $L^\circ$. Finally, the construction of the $\bbT_{L^\circ}$-toric varieties corresponding to $\Gamma'$-admissible fans shows that the inclusion $\Sigma\subset\nbar{\Sigma}$ induces a $\bbT_{L^\circ}$-equivariant open immersion $\calX_{L^\circ}\into\nbar{\calX}$.
\end{proof}

\subsection{Projective $\T$-toric varieties whose normalizations are not of finite type}\label{subsec:BadNormalization}

We now give an explicit example which shows that there are projective $\T$-toric varieties whose normalizations are not finite type, so long as the value group $\Gamma$ is neither discrete nor divisible. In particular, this shows that taking the normalization of an arbitrary equivariant completion of a normal $\bbT$-toric variety will not always give a normal equivariant completion of the original variety because this normalization may not be of finite type.

Throughout this subsection we assume that $\Gamma$ is neither discrete nor divisible. Hence there is a positive integer $r$ such that $r\Gamma$ is not all of $\Gamma$, so we can pick some positive $\gamma\in\Gamma\sdrop r\Gamma$. Fix a positive integer $n$ and let $M=\Z^n$ so $\T=\G_{m,K^\circ}^{n}$. We identify $\T$ with a subtorus of $\G_{m,K^\circ}^{2n+1}/\G_{m,K^\circ}\subset\pp_{K^\circ}^{2n}$ via the morphism $(z_1,\ldots,z_n)\mapsto(1:z_1^r:z_1^{r+1}:z_2^r:z_2^{r+1}:\cdots:z_n^r:z_n^{r+1})$. In particular, this gives an action of $\T$ on $\pp_{K^\circ}^{2n}$. Fix some $\pi\in K^\circ$ with valuation $\gamma$, let $y:=(\pi:1:1:\cdots:1)\in\pp^{2n}(K)$, and let $\calY:=\nbar{T\cdot y}$, the closure in $\pp_{K^\circ}^{2n}$ of the orbit of $y$. Note that $\calY$ is a $\T$-toric variety. For the sake of concreteness we point out that a routine but tedious computation shows that $\calY$ is the subscheme of $\pp_{K^\circ}^{2n}$ cut out by the ideal $(\pi x_{2i-1}^{r+1}-x_{2i}^rx_0,x_{2i-1}^{r+1}x_{2j}^r-x_{2j-1}^{r+1}x_{2i}^r\mid i,j=1,\ldots,n)$.

\begin{thm}\label{thm:BadNormalization}
The normalization of $\calY$ is not of finite type over $K^\circ$.
\end{thm}\begin{proof}
Note that $A:=K^\circ[\frac{1}{\pi}z_1^r,\frac{1}{\pi}z_1^{r+1},\frac{1}{\pi}z_2^r,\frac{1}{\pi}z_2^{r+1}\ldots,\frac{1}{\pi}z_n^r,\frac{1}{\pi}z_n^{r+1}]$ is the coordinate ring of the open affine subset of $\calY$ where the $x_0$ coordinate of $\pp_{K^\circ}^{2n}$ does not vanish. Hence it suffices to show that the integral closure $\widetilde{A}$ of $A$ in $K(z_1,\ldots,z_n)$, the field of fractions of $A$, is not finitely generated as a $K^\circ$-algebra.

We claim that $\widetilde{A}$ is $B:=K^\circ[\{\lambda z^I\mid \lambda\in K, I=(i_1,\ldots,i_n)\in\Z^n, |I|\frac{\gamma}{r}+v(\lambda)\geq0, i_j\geq0\}]$ where $|I|:=i_1+\cdots+i_n$. First, note that any term $\lambda z^I$ with $|I|\frac{\gamma}{r}+v(\lambda)\geq0$ and $i_j\geq0$ is in $\widetilde{A}$ because it is a zero of the monic polynomial $t^r-(\lambda^r\pi^{|I|})\prod_{j=1}^n(\frac{1}{\pi}z_j^r)^{i_j}$ in $A[t]$. Thus $B\subset\widetilde{A}$. For the other inclusion, it suffices to show that $B$ is integrally closed in $K(z_1,\ldots,z_n)$. The valuation $\sum_{\ell=1}^m \lambda_\ell z^{I_\ell}\mapsto\min\{|I_\ell|\frac{\gamma}{r}+v(\lambda_\ell)\mid \ell=1,\ldots,m\}$ on $K[z_1,\ldots,z_n]$ extends to a valuation $w$ on $K(z_1,\ldots,z_n)$ and $B$ is the intersection of $K[z_1,\ldots,z_n]$ and the valuation ring of $w$. Thus $B$, as the intersection of two integrally closed subrings of $K(z_1,\ldots,z_n)$, is integrally closed in $K(z_1,\ldots,z_n)$.

To see that $\widetilde{A}$ is not finitely generated as a $K^\circ$-algebra, note that because $\widetilde{A}$ is $(\Z_{\geq0})^n$-graded, if it had a finite generating set then it would have one consisting of terms $\lambda z^I$. But $\widetilde{A}$ cannot have such a generating set because the homogeneous component of $\widetilde{A}$ of multi-degree $I=(1,0,0,\ldots,0)$ is $\{\lambda z_1\mid \lambda\in K, v(\lambda)\geq-\frac{\gamma}{r}\}$ which is not finitely generated as a $K^\circ$-module.
\end{proof}

\begin{remark}
Theorem \ref{thm:BadNormalization} can also be proven using the ideas of \cite[Proposition 2.9]{Soto}. Namely, the proof of that proposition can be used to show that the normalization $\widetilde{\calY}$ of $\calY$ is the $\T$-toric scheme corresponding to the fan whose maximal faces are
\begin{align*}
&\left\{ (w,t)\in N_{\R}\times\R_{\geq0}\ \middle|\ w_j\geq\frac{\gamma}{r}t\text{ for }j=1,\ldots,n\right\},\\
&\left\{(w,t)\in N_{\R}\times\R_{\geq0}\ \middle|\ 0\leq w_i\leq \frac{\gamma}{r}t, w_i\leq w_j\text{ for }j\neq i\right\}\text{ for }i=1,\ldots,n,\\
\intertext{and}
&\left\{(w,t)\in N_{\R}\times\R_{\geq0}\ \middle|\ w_i\leq 0, w_i\leq w_j\text{ for }j\neq i\right\}\text{ for }i=1,\ldots,n.
\end{align*}
Then \cite[Proposition 6.9]{GublerGuideTrop} can be used to deduce that $\widetilde{\calY}$ is not of finite type over $K^\circ$.
\end{remark}

We can view the generic fiber $\calX=\T_\eta$ of $\T$ as a normal $\T$-toric variety and $\calY$ as an equivariant completion of $\calX$. Then we have that $\calY$ is an equivariant completion of $\calX$ whose normalization is not a normal equivariant completion of $\calX$.

\subsection{Semistable $\T$-toric varieties with no semistable equivariant completion}\label{subsec:SemistableTToricVars}

Recall that a scheme $\calX$ over $K^\circ$ is called \emph{semistable} if every point of $\calX$ has an \'{e}tale neighborhood which admits an \'{e}tale morphism to a \emph{model semistable $K^\circ$-scheme} $\Spec K^\circ[x_0,\ldots,x_n]/(x_0\cdots x_m-\pi)$ with $0\leq m\leq n$ and $\pi\in K^\circ\sdrop\{0\}$. We say that $\calX$ is \emph{strictly semistable} if the \'{e}tale neighborhoods as above can be chosen to be Zariski open neighborhoods.

We note that semistable $K^\circ$-schemes are normal. Towards seeing this we first show that any model semistable $K^\circ$-schemes is normal by exhibiting it as the normal $\T$-toric variety corresponding to a $\Gamma$-admissible cone. 
Given a model semistable $K^\circ$-scheme $\Spec K^\circ[x_0,\ldots,x_n]/(x_0\cdots x_m-\pi)$ let $\gamma=v(\pi)$ and $N=\Z^n$. Let $\sigma=\{(w_1,\ldots,w_n,t)\in N_{\R}\times\R_{\geq0}\mid w_i\geq 0(\forall i=1,\ldots,n), w_1+\cdots+w_m\leq\gamma t\}$. 
Then a quick computation of the vertices of $\sigma\cap(N_{\R}\times\{1\})$ and their local cones shows that $K[M]^\sigma=K^\circ[\pi x_1^{-1}\cdots x_m^{-1},x_1,\ldots,x_n]\cong K^\circ[x_0,\ldots,x_n]/(x_0\cdots x_m-\pi)$, so $\Spec K^\circ[x_0,\ldots,x_n]/(x_0\cdots x_m-\pi)\cong \calU_\sigma$ is normal. 
Finally, the normality of arbitrary semistable $K^\circ$-schemes follows from the normality of model semistable $K^\circ$-schemes because normality is local in the \'{e}tale topology; see \cite[Tags 034F and 0347]{Stacks}.

We will see that the $\T$-toric variety corresponding to the fan constructed in Example \ref{example:FanWOASimplicialCompletion} is strictly semistable but does not admit a semistable equivariant completion. In order to show that such completions do not exist we will use the following proposition.

\begin{prop}\label{prop:SemistableGivesBoundedSimplices}
Let $\calX$ be a semistable $\T$-toric variety and let $\Sigma$ be the corresponding fan in $N_{\R}\times\R_{\geq0}$. Let $\Pi$ be the restriction of $\Sigma$ to $N_{\R}\times\{1\}$. Then every bounded face of $\Pi$ is a simplex.
\end{prop}\begin{proof}
Because $\calX$ is semistable, the special fiber $\calX_s$ of $\calX$ is a normal crossings variety over the residue field $k$ of $K$. In particular, the intersection of any $d$ irreducible components of $\calX_s$ is either empty or has pure codimension $d-1$ in $\calX_s$. 
By \cite[Section 7.9 and Proposition 8.8]{GublerGuideTrop} there is an order-reversing bijection from the set of open faces of $\Sigma$ and the set of torus orbits of $\calX$ which we write as $\tau\mapsto Z_\tau$. 
For any open face $\tau$ of $\Sigma$, $\dim\tau=\codim(Z_\tau,\calX)$. We also have that $Z_\tau\subset \calX_s$ if and only if $\tau\subset N_{\R}\times\R_{>0}$, and in this case $\dim\tau=\codim(Z_\tau,\calX)=\codim(Z_\tau,\calX_s)+1$. 
Thus there is an order reversing bijection between the faces of $\Pi$ and the torus orbits contained in $\calX_s$, denoted $P\mapsto Z_P$, such that $\dim P=\codim(Z_P,\calX_s)$. 
For $P\in\Pi$ let $W_P$ denote the closure of $Z_P$. Then the irreducible components of $\calX_s$ are exactly the $W_v$ with $v$ a vertex of $\Pi$, and if $v_1,\ldots,v_d$ are vertices of $\Pi$ then the intersection $W_{v_1}\cap\cdots\cap W_{v_d}$ is $W_{P}$ where $P$ is the smallest face of $\Pi$ containing all of $v_1,\ldots,v_d$, or is empty if no such face of $\Pi$ exists. In particular, if a bounded face $P\in\Pi$ has exactly $d$ vertices $v_1,\ldots,v_d$ then we must have 
$$d-1=\codim(W_{v_1}\cap\cdots\cap W_{v_d},\calX_s)=\codim(W_P,\calX_s)=\dim P$$
so $P$ must be a simplex.
\end{proof}

\begin{remark}
The proof of Proposition \ref{prop:SemistableGivesBoundedSimplices} also shows that if $\Sigma$ is a $\Gamma$-admissible fan in $N_{\R}\times\R_{\geq0}$ corresponding to a $\T$-toric variety $\calX$ with special fiber $\calX_s$ and $\Pi$ is the restriction of $\Sigma$ to $N_{\R}\times\{1\}$ then the set of bounded faces of $\Pi$ gives a suitable notion of a dual complex for the variety $\calX_s$ which agrees with the usual definition of the dual complex in the case when $\calX_s$ is a simple normal crossings variety.
\end{remark}

\begin{thm}
Assume that the rational rank of $\Gamma$ is at least two, i.e., that $\Gamma$ contains two $\Q$-linearly independent elements. Let $\T=\G_{m,K^\circ}^2$ and let $\calX=\calY(\Sigma)$ where $\Sigma$ is the fan constructed in Example \ref{example:FanWOASimplicialCompletion}. Then $\calX$ is strictly semistable but does not have any semistable equivariant completion.
\end{thm}\begin{proof}
We first show that $\calX$ is strictly semistable. Consider any cone $\sigma$ in $N_{\R}\times\R_{\geq0}$ such that $\sigma\cap(N_{\R}\times\{1\})=P\times\{1\}$ with $P$ a $\Gamma$-rational line segment with endpoints in $N_{\Gamma}$. Let $\calU_\sigma$ be the corresponding normal affine $\T$-toric variety. 
A quick computation using the local cones of $P$ shows that $\calU_\sigma$ is isomorphic to $\Spec K^\circ[x_0,x_1,x_2^{\pm1}]/(x_0x_1-\pi)$ where $\pi\in K^\circ$ is such that $v(\pi)$ is the lattice length of $P$. 
Since the maximal cones of $\Sigma$ are all of the form described above, this shows that $\calX$ is covered by Zariski open subsets which are isomorphic to Zariski open subsets of model semistable $K^\circ$-schemes, so $\calX$ is strictly semistable.

Now suppose $\nbar{\calX}$ were a semistable equivariant completion of $\calX$. Then $\nbar{\calX}$ would give a $\Gamma$-admissible completion $\nbar{\Sigma}$ of $\Sigma$. Let $\Pi$ and $\nbar{\Pi}$ be the restrictions of $\Sigma$ and $\nbar{\Sigma}$, respectively, to $N_{\R}\times\{1\}$. Then by Proposition \ref{prop:SemistableGivesBoundedSimplices}, all of the bounded faces of $\nbar{\Pi}$ would be simplices. But in Example \ref{example:FanWOASimplicialCompletion} we showed that no such completion of $\Pi$ could exist. Hence $\calX$ does not have any semistable equivariant completion.
\end{proof}

%
%

\bibliographystyle{plain}
\bibliography{NormalCompletions}

\end{document}